\documentclass[oneside]{gtpart}
\agtart
\usepackage{etex}
\usepackage[a4paper]{geometry}
\usepackage{pinlabel}
\usepackage{amssymb,amsfonts,amsmath}
\usepackage{comment} 
\usepackage[all,arc]{xy}
\usepackage{enumerate}
\usepackage{float}
\usepackage[caption = false]{subfig}
\usepackage{mathrsfs,mathtools}
\usepackage{todonotes,booktabs}
\usepackage{stmaryrd}
\usepackage{marvosym}
\usepackage{graphicx}
\usepackage{pdflscape}
\tikzset{circlen/.code={\def\circlen@n{#1}\pgfkeysalso{shape=circlen@shape}}}
 \newcommand{\E}{\mathcal{E}}
\usepackage{bbm}

\newcommand{\bu}{\bar{u}}

\usepackage{tikz}
\usepackage{tikz-cd}
\usetikzlibrary{trees}
\usetikzlibrary[shapes]
\usetikzlibrary[arrows]
\usetikzlibrary{patterns}
\usetikzlibrary{fadings}
\usetikzlibrary{backgrounds}
\usetikzlibrary{decorations.pathreplacing}
\usetikzlibrary{decorations.pathmorphing}
\usetikzlibrary{positioning}

\arxivreference{1810.05439}

\usepackage{xcolor} 
\usepackage{graphicx}

\PassOptionsToPackage{pagebackref}{hyperref}

\definecolor{dark-red}{rgb}{0.5,0.15,0.15}
\definecolor{dark-blue}{rgb}{0.15,0.15,0.6}
\definecolor{dark-green}{rgb}{0.15,0.6,0.15}
\hypersetup{
    colorlinks, linkcolor=dark-red,
    citecolor=dark-blue, urlcolor=dark-green
}

\newcommand{\mE}{(E_n^{\vee})_*}

\newcommand*{\backref}[1]{}
\newcommand*{\backrefalt}[4]{%
  \ifcase #1 %
No citations.
  \or
(cit. on p. #2).%
  \else
(cit on pp. #2).%
  \fi%
}
\usepackage{subfiles}

\usepackage[nameinlink,capitalise,noabbrev]{cleveref}

\newtheorem{thm}{Theorem}[section]
\newtheorem{cor}[thm]{Corollary}
\newtheorem{prop}[thm]{Proposition}
\newtheorem{lem}[thm]{Lemma}

\newtheorem*{thm*}{Theorem}

\theoremstyle{definition}
\newtheorem{defn}[thm]{Definition}

\newtheorem{ex}[thm]{Example}

\theoremstyle{remark}
\newtheorem{rem}[thm]{Remark}

\makeatletter
\let\c@equation\c@thm
\makeatother
\numberwithin{equation}{section}

\DeclareMathOperator{\Aut}{Aut}
\DeclareMathOperator{\Sp}{Sp}
\DeclareMathOperator{\Hom}{Hom}

\DeclareMathOperator{\colim}{colim}

\DeclareMathOperator{\cC}{\mathcal{C}}

\DeclareMathOperator{\cE}{\mathcal{E}}

\DeclareMathOperator{\Fun}{Fun}

\renewcommand{\Q}{\mathbb{Q}}

\DeclareMathOperator{\id}{id}

\DeclareMathOperator{\Pic}{Pic}

\DeclareMathOperator{\unit}{\mathbbm{1}}

\DeclareMathOperator{\Gal}{Gal}

\newcommand{\G}{\mathbb{G}}

\renewcommand{\cal}{\mathcal}
\newcommand{\xr}{\xrightarrow}

\renewcommand{\R}{\mathbb{R}}
\renewcommand{\Z}{\mathbb{Z}}

\Crefname{figure}{Figure}{Figures}
\Crefname{assu}{Assumption}{Assumptions}
\Crefname{lem}{Lemma}{Lemmas}
\Crefname{thm}{Theorem}{Theorems}
\Crefname{prop}{Proposition}{Propositions}

\newcommand{\recollement}[5]{
\xymatrix{{#1} \ar[r]|-{#2} & #3 \ar[r]|-{#4} \ar@<1ex>[l]^-{{#2}_!} \ar@<-1ex>[l]_-{{#2}^*} & #5, \ar@<1ex>[l]^-{{#4}!} \ar@<-1ex>[l]_-{{#4}^*}
}}
\let\lim\relax

\DeclareMathOperator{\lim}{lim}

\newcommand{\F}{\mathbb{F}}

\DeclareMathOperator{\HFPSS}{\text{HFPSS}}

\DeclareMathOperator{\map}{map}
\DeclareMathOperator{\coker}{coker}
\DeclareMathOperator{\kernel}{ker}

\usepackage{enumitem}
\Crefname{fig}{Figure}{Figures}

\newcommand{\pics}{\mathfrak{pic}}

\title[Picard groups and duality for Real Morava $E$-theories]{Picard groups and duality for Real Morava $E$-theories}

\author{Drew Heard}
\givenname{Drew}
\surname{Heard}
\address{Department of Mathematical Sciences, Norwegian University of Science and Technology, Trondheim}
\email{drew.k.heard@ntnu.no}
\urladdr{https://folk.ntnu.no/drewkh/}

\author{Guchuan Li}
\address{Department of Mathematics, University of Michigan, Ann Arbor, MI 48109}
\email{guchuan@umich.edu}

\author{XiaoLin Danny Shi}
\address{Department of Mathematics, University of Chicago, Chicago IL 60637}
\email{dannyshixl@gmail.com}

\date{\today}

\begin{document}

\begin{abstract}
We show, at the prime 2, that the Picard group of invertible modules over $E_n^{hC_2}$ is cyclic.  Here, $E_n$ is the height $n$ Lubin--Tate spectrum and its $C_2$-action is induced from the formal inverse of its associated formal group law.  We further show that $E_n^{hC_2}$ is Gross--Hopkins self-dual and determine the exact shift. Our results generalize the well-known results when $n = 1$. 
\end{abstract}
\maketitle

\setcounter{tocdepth}{1}

\section{Introduction}
Let $E_n$ be the height-$n$ Lubin--Tate spectrum with coefficient ring 
\[
\pi_*(E_n) \cong \mathbb{W}(\F_{p^n})[\![u_1,\ldots,u_{n-1}]\!][u^{\pm 1}]. 
\]
This spectrum is constructed by using the theory of formal groups (see \Cref{sec:chromintro} for more details).  For this paper, we will focus on the case when the prime $p$ is $2$.  At $p = 2$, the spectrum $E_n$ has a $C_2$-action that is induced from the formal inverse of its associated formal group law via the Goerss--Hopkins--Miller theorem \cite{goerss_hopkins,Rezk}.  Our main object of interest in this work will be the homotopy fixed point spectrum $E_n^{hC_2}$. 

When $n = 1$, there is an equivalence between $E_1^{hC_2}$ and $KO^{\wedge}_2$, the $2$-completed real $K$-theory.  The real Bott periodicity implies that $E_1^{hC_2}$ is 8-periodic.  In \cite{ms_pic}, Mathew and Stojanoska have shown that the Picard group $\Pic(E_1^{hC_2})$ of invertible $E_1^{hC_2}$-modules is generated by the suspension $\Sigma E_1^{hC_2}$ and is isomorphic to $\Z/8$. 

Using $C_2$-equivariant homotopy theory, Hahn and Shi \cite{hahn_shi} have recently computed the homotopy fixed point spectral sequence of $E_n^{hC_2}$ for all heights $n \geq 1$.  In particular, they show that the spectrum $E_n^{hC_2}$ is $2^{n+2}$-periodic.  

In this paper, we generalize the Picard group result in \cite{ms_pic} to all higher heights. This is the first systematic result that applies for all heights. 
\theoremstyle{plain}
\newtheorem*{thm:picardcalc}{\Cref{thm:picardcalc}}
\begin{thm:picardcalc}
For $n \geq 1$, the Picard group $\Pic(E_n^{hC_2}) \cong \Z/2^{n+2}$ and is generated by $\Sigma E_n^{hC_2}$. 
\end{thm:picardcalc}

In \cite{hopkins_gross}, Gross and Hopkins investigated a certain type of duality in the $K(n)$-local category that is analogous to the Grothendieck--Serre duality in algebraic geometry.  Namely, for a $K(n)$-local spectrum $X$, define 
$$IX = F(M_nX,I_{\Q/\Z}).$$ 
Here, $M_nX$ denotes the $n$-th monochromatic layer of $X$ and $I_{\Q/\Z}$ is the Brown--Comenetz dualizing spectrum.  When $n = 1$, it is well-known that $IE_1^{hC_2} \simeq \Sigma^5E_1^{hC_2}$ (see \cite[Sec.~9]{anderson_ko} for example).  Using \Cref{thm:picardcalc}, we show that more generally, $IE_n^{hC_2}$ is always self-dual up to a shift.  More specifically, we prove that $IE_n^{hC_2} \simeq \Sigma^{4+n}E_n^{hC_2}$.  In fact, this equivalence can be refined to a $C_2$-equivariant equivalence: 

\newtheorem*{thm:grosshopkinsequiv}{\Cref{thm:grosshopkinsequiv}}
\begin{thm:grosshopkinsequiv}
  The Gross--Hopkins dual $IE_n$ is $C_2$-equivariantly equivalent to $\Sigma^{4+n}E_n$. 
\end{thm:grosshopkinsequiv}
In the process of proving Theorem~\Cref{thm:grosshopkinsequiv}, we compute the Morava module of $IE_n^{hF}$ for any finite subgroup $F$ of the Morava stabilizer group.  Our computation is valid for all primes $p$.  Although we do not use this in our main computations, we include this as it may be of independent interest to the readers. 

Finally, in \Cref{sec:exoticpic}, we show that our computations, when combined with work of Beaudry--Goerss--Hopkins--Stojanoska \cite{bghs} and Barthel--Beaudry--Goerss--Stojanoska \cite{bbgs}, imply that the exotic part of the $K(n)$-local Picard group is always non-trivial when $p=2$, see \Cref{prop:exotic}, and that the group has an element whose order has a lower bound that grows exponentially with respect to $n$.


\subsection*{Outline of proof} We now provide an outline of proof for our main theorems.  To compute the Picard group, we use the techniques that are developed by Mathew and Stojanoska in \cite{ms_pic}.  More specifically, there is a descent spectral sequence that is associated to the Galois extension $E_n^{hC_2} \to E_n$.  Using this spectral sequence, we compute the homotopy groups of a certain Picard spectrum $\pics(E_n^{hC_2})$, which has the property that $\pi_0\pics(E_n^{hC_2}) \cong \Pic(E_n^{hC_2})$.  In a certain range, the Picard spectral sequence concides with the usual homotopy fixed point spectral sequence for $E_n^{hC_2}$.  Moreover, Mathew and Stojanoska have shown that in a more restricted range, differentials can be imported from the homotopy fixed point spectral sequence to the Picard spectral sequence.

In the additive spectral sequence, there are $n$ non-trivial differentials.  Using formulas in \cite{ms_pic}, we see that each of these differentials has kernel $\Z/2$ in the 0-stem of the Picard spectral sequence.  This, together with some low dimensional calculations, shows that the Picard group has order at most $2^{n+2}$.  However, we also know that this is the lower bound by periodicity.  Therefore, the Picard group must have order exactly $2^{n+2}$. 

We will then give two proofs for the Gross--Hopkins dual $IE_n^{hC_2}$.  The first proof relies extensively on techniques from $C_2$-equivariant homotopy theory, as well as a general result, which states that $I(E_n^{hF}) \in \Pic(E_n^{hF})$ for any finite subgroup of the Morava stabilizer group (\Cref{prop:grosshopkins_pic}).  The second proof is more computational in nature.  We set up a series of spectral sequences and completely compute the homotopy fixed point spectral sequence for $IE_n^{hC_2}$.  This computation uses a technique called the ``geometric boundary theorem'', which has been studied by Behrens in \Cref{app:gbt}. 

\subsection*{Acknowledgments}
The authors would like to thank Tobias Barthel, Agnes Beaudry, Mark Behrens, Irina Bobkova, Sanath Devalapurkar, Paul Goerss, Jeremy Hahn, Achim Krause, Lennart Meier, Tomer Schlank, Vesna Stojanoska, Guozhen Wang, Craig Westerland, and Zhouli Xu for helpful conversations.  The first author would like to thank Haifa University and Regensburg University for their hospitality.  This research was partially supported by the DFG grant SFB 1085 "Higher invariants".  All the spectral sequence pictures in this paper are created using Hood Chatham's spectral sequence package. Finally, we thank the referee for many helpful comments and suggestions. 
\section{Background}
We begin with the necessary background on Lubin--Tate spectra, Gross--Hopkins duality, and $C_2$-equivariant homotopy theory. 
\subsection{Lubin--Tate spectra and chromatic homotopy}\label{sec:chromintro}
We recall (e.g., from \cite{rezk_notes}) that given a perfect field $k$ of characteristic $p > 0$ and a formal group $F$ over $k$ of height $n$, then there exists a Lubin--Tate spectrum $E(k,F)$ with formal group the universal deformation of $F$. The coefficients of $E(k,F)$ are given by 
\[
\pi_*E(k,F) \cong \mathbb{W}(k)[\![u_1,\ldots,u_{n-1}]\!][u^{\pm 1}]
\]
where $|u_i| = 0$ and $|u| = 2$, and $\mathbb{W}(k)$ denotes the Witt vectors of the field $k$. Note that this is a complete local ring with maximal ideal $I_n = (p,u_1,\ldots,u_{n-1})$. For concreteness, we always work with $k = \F_{p^n}$ and $F$ the Honda formal group law with $p$-series $[p](x) = x^{p^n}$. We denote the associated spectrum by $E_n$.  Note that our results are easily generalized to other forms of Lubin--Tate spectra as well because any $E(k, F)$ admits a Real-orientation by \cite{hahn_shi}.  As a result, its $C_2$-homotopy fixed point spectral sequence is always regular in the sense of \cite[Definition~6.1]{MeierRegularHFPSS}.

Let $\mathbb{S}_n$ denote the group of automorphisms of the Honda group law over $\F_{p^n}$, and let $\mathbb{G}_n = \mathbb{S}_n\rtimes \Gal(\F_{p^n}/\F_p)$. Lubin--Tate theory implies that $\G_n$ acts on $(E_n)_*$ and the Goerss--Hopkins--Miller theorem improves this to show that $\G_n$ acts on the spectrum $E_n$ via $E_{\infty}$-ring maps \cite{goerss_hopkins}. In this paper we are particularly interested in the $C_2$-action on $E_n$ when $p = 2$, arising from the formal inverse of the Honda formal group law, and the associated homotopy fixed point spectrum $E_n^{hC_2}$. Here the $C_2$-action on $(E_n)_*$ is particularly simple; if we let $\tau$ denote the generator of $C_2$, then $\tau_*(u_i) = u_i$ and $\tau_*(u^k) = (-1)^ku^k$ (see for example the proof of \cite[Lemma~1.33]{bobkova_goerss}). As noted in the introduction, when $n = 1$ there is an equivalence between $E_1^{hC_2}$ and the $2$-complete real $K$-theory. 

Let $L_n$ denote the Bousfield localization functor with respect to $E_n$. The Bousfield class $\langle E_n \rangle$ admits a decomposition 
\[
\langle E_n \rangle = \langle K(0) \vee K(1) \vee \cdots \vee K(n) \rangle 
\]
where $K(n)$ denotes the $n$-th Morava $K$-theory \cite{hs_99}. Given a $K(n)$-local spectrum $X$, the natural version of homology to define is 
\[
\mE X = \pi_*L_{K(n)}(E_n \otimes X),
\]
This is not always $I_n$-adically complete as an $(E_n)_*$-module (although see \cite[Sec.~2]{ghmr} for conditions that ensure this), but it is always $L$-complete, see \cite[Prop.~8.4]{hs_99}, where $L$-completion denotes the zero-th derived functor of $I_n$-adic completion. It also obtains an action of $\G_n$ via the action on $(E_n)_*$, which is twisted, in the sense that if $g \in \G_n$, $a \in (E_n)_*$ and $x \in \mE X$, then 
\[
g(ax) = g(a)g(x). 
\]
We call any such $L$-complete $(E_n)_*$-module with this twisted $\G_n$-action a \emph{Morava module}, see \cite[Sec.~2]{ghmr}. 

If $M$ is a Morava module, we let $\map_{cts}(\G_n,M)$ denote the group of continuous maps from $\G_n \to M$, with $(E_n)_*$-action given by 
\[
(a\phi)(x) = a\phi(x)
\]
for $a \in (E_n)_*$ and $\phi \colon \mathbb{G}_n \to M$. If we give $\map_{cts}(\G_n,M)$ the diagonal $\G_n$-action defined by 
\[
(g\phi)(x) = g\phi(g^{-1}x)
\]
then this obtains the structure of a Morava module.  For example, there is an isomorphism $\mE E_n \cong \map_{cts}(\G_n,(E_n)_*)$ of Morava modules, see \cite[Thm.~2]{dh_04}. In fact, more generally, we have $\mE (E_n^{hF}) \cong \map_{cts}(\G_n/F,(E_n)_*)$, see \cite[Prop.~6.3]{dh_04}. 
\begin{rem}
  If $M$ is an $(E_n)_*$-module with a compatible action of a finite subgroup $F \subset \G_n$, we say that $M$ is an $(E_n)_*[F]$-module. For consistency, we could also call Morava modules $(E_n)_*[\![\G_n]\!]$-modules, but the terminology Morava module seems well-established by now, and so we stick with that.  
\end{rem}

\subsection{Gross--Hopkins duality}
Let $I_{\Q/\Z}$ denote the Brown--Comenetz spectrum \cite{bc_dual}, and let $I_{\Q/\Z}X = F(X,I_{\Q/\Z})$ denote the Brown--Comenetz dual of a spectrum $X$, characterized by $\pi_*I_{\Q/\Z}X \cong \Hom(\pi_{-*}X,\Q/\Z)$. For a $K(n)$-local spectrum $X$, the Gross--Hopkins dual of $X$, denoted $IX$, is defined as $IX = F(M_nX,I_{\Q/\Z})$, i.e., as the Brown--Comenetz dual of $M_nX$, where $M_nX$ is the fiber of the natural map $L_nX \to L_{n-1}X$.  We write $I = IS^0$; it follows that $IX = F(X,I)$. Following Gross and Hopkins, Strickland \cite{strickland_gross} has studied the properties of $I$ in some detail.  In order to state what we need, we recall that there is a determinant map 
\[
\det \colon \G_n \to \Z_p^{\times},
\]
see \cite[Sec.~1.3]{ghmr}. If $M$ is a Morava module, then we write $M \langle \det \rangle$ for the same $(E_n)_*$-module with the $\G_n$-action twisted by $\det$. With this is mind, we have the following \cite[Thm.~2]{strickland_gross}, which relies extensively on the work of Gross and Hopkins \cite{hopkins_gross}.  
\begin{thm}[Strickland]\label{thm:moravabc}
  There is an isomorphism of Morava modules 
  \[
(E^{\vee}_n)_* I \cong \Sigma^{n^2-n} (E_n)_* \langle \det \rangle. 
  \]
\end{thm}
This implies that $I$ is invertible in the $K(n)$-local category. In fact, when $p \gg n$, this implies that $I \simeq\Sigma^{n^2-n} S\langle \det \rangle$, where $S \langle \det \rangle$ is the determinantal sphere, a $K(n)$-locally invertible spectrum with $\mE S \langle \det \rangle \cong (E_n)_*\langle \det \rangle$. We return to this connection more in \Cref{sec:exoticpic}. 
\begin{rem}\label{rem:bracketing}
  Note that for any finite subgroup $F \subset \mathbb{G}_n$ there is an equivalence $I(E_n^{hF}) \simeq (IE_n)^{hF}$. Indeed, since the Tate spectrum $E_n^{tF}$ vanishes $K(n)$-locally, see \cite{gs_ate}, and because $I$ is $K(n)$-local, we obtain the following chain of equivalences: 
$$I(E_n^{hF}) \simeq F(E_n^{hF},I) \simeq F((E_n)_{hF},I) \simeq F(E_n, I)^{hF} \simeq (IE_n)^{hF}.$$

Therefore, there is no ambiguity if we simply write $IE_n^{hF}$.
\end{rem}

 We will also need to know that there is an isomorphism
\begin{equation}\label{eq:detinen}
\pi_*IE_n \cong \Sigma^{-n}(E_n)_*\langle \det \rangle,
\end{equation}
which is proved by Strickland \cite{strickland_gross}.

In \cite[Prop.~2.1]{bbs_gross} the authors show that the homotopy fixed point spectral sequence computing $\pi_*(I_{\Q/\Z}X)^{hH}$ is Pontryagin dual to the homotopy orbits spectral sequence computing $\pi_*(X_{hH})$. We restate this here, where $D_{\Q/\Z}(-)$ denotes the Pontryagin dual of an abelian group.  
\begin{lem}[Barthel--Beaudry--Stojanoska]\label{lem:qzdual}
Let $G$ be a finite group and $X$ a $G$-spectrum.  Then the homotopy fixed point spectral sequence of $I_{\Q/\Z}X$ is Pontryagin dual to the homotopy orbit spectral sequence of $X$.  More precisely, let $\E_r^{*,*}$ be the $r^{th}$ page of the HFPSS$((I_{\Q/Z}X)^{hG})$ with differentials 
$$d_r \colon \E_r^{s,t} \rightarrow \E_r^{s+r,t+r-1},$$
and $\E^r_{*,*}$ be the $r^{th}$ page of HOSS$(X_{hG})$ with differentials 
$$\overline{d_r} \colon \E^r_{s,t} \rightarrow \E^r_{s-r,t+r-1}.$$ 
Then there is a natural isomorphism between each page and between all the differentials:
$$\E_r^{s,t} \cong D_{\Q/Z}(\E^r_{s,-t}), \quad D_{\Q/\Z}(\overline{d_r})=d_r.$$     
\end{lem}

\subsection{The monochromatic layer of the sphere and Morava \texorpdfstring{$E$}{E
}-theory}\label{subsection:mono}
Recall that the Gross--Hopkins dual of Morava $E$-theory is the Brown--Comenetz dual of $M_nE_n$. In this section, we give an alternative description of the monochromatic layer of $E$-theory.

To begin, let $I_k$ be the ideal $(p,v_1,\ldots,v_{k-1})$ in $(E_n)_*$, where 
\[
v_k = 
\begin{cases}
  p & k = 0 \\
  u_ku^{1-p^k} & 1 \le k < n \\
  u^{1-p^n} & k=n. 
\end{cases}
\]
We make the convention that $I_0$ is the zero ideal, so that $E_n/I_0 \simeq E_n$. 

We recall that we can inductively define spectra $E_n/I_k^{\infty}$ by cofiber sequences
\begin{equation}\label{eq:cofibermnen}
E_n/I_k^{\infty} \to v_k^{-1}E_n/I_k^{\infty} \to E_n/I_{k+1}^{\infty}
\end{equation}
Applying $M_n$ to these cofiber sequences, and noting that $M_n(v_k^{-1}E_n/I_k^{\infty}) \simeq \ast$ is trivial for $k<n$, and that $M_n(E_n/I_n^{\infty}) \simeq E/I_n^{\infty}$, we obtain an equivalence $\Sigma^{-n}E_n/I_n^{\infty} \xr{\simeq} M_nE_n$. The main goal of this section is to prove that this is an equivariant equivalence. 
\begin{prop}\label{prop:equivmonoequiv}
  There is an equivalence 
  \[
\Sigma^{-n}E_n/I_n^{\infty}\xr{\simeq} M_nE_n 
  \]
  compatible with the $\mathbb{G}_n$-action on both sides. 
\end{prop}
To make this precise, we need to specify how $\mathbb{G}_n$ acts. We will in fact show that $M_nE_n \simeq E_n \otimes M_nS^0$ and $\Sigma^{-n}E_n/I_n^{\infty} \simeq \Sigma^{-n}E_n \otimes S^0/I_n^{\infty}$, where $S^0/I_n^{\infty}$ is defined as the homotopy limit of the generalized Moore spectra of type $n$ (see \cite[Sec.~4.1]{hs_99}). Here the action of $\mathbb{G}_n$ is given by acting only on $E_n$, where it acts in the usual way. Moreover, the map will be the identity on $E_n$, and hence equivariant.

\begin{proof}[Proof of \Cref{prop:equivmonoequiv}]
By construction, the spectrum $S^0/I_n^\infty$ fits into the telescope tower of the $p$-local sphere spectrum $S^0$ as follows:

   \[
      \setlength\mathsurround{0pt}
    \begin{tikzcd}
      S^0 \arrow{d} & S^0/p^\infty  \arrow{d} \arrow[dashrightarrow]{l}& S^0/(p^\infty,v_1^\infty) \arrow{d} \arrow[dashrightarrow]{l} & S^0/I_3^\infty \arrow[dashrightarrow]{l} \arrow{d} & S^0/I_4^\infty \arrow[dashrightarrow]{l} \arrow{d}& \cdots \arrow[dashrightarrow]{l}\\
      p^{-1}S^0 \arrow{ur}   & v_1^{-1}S/p^\infty \arrow{ur} & v_2^{-1}S/(p^\infty, v_1^\infty) \arrow{ur} & v_3^{-1}S^0/I_3^\infty \arrow{ur} & v_3^{-1}S^0/I_4^\infty \arrow{ur}
    \end{tikzcd}
    \]
 In this diagram, each sequence
  \[
S^0/I_k^{\infty} \to v_k^{-1}S^0/I_k^{\infty} \to S^0/I_{k+1}^{\infty}
  \]
  is a cofiber sequence, so that the dashed arrows indicate that there is a map from the source to a suspension of the target. 

  After applying $\pi_*(E_n \otimes -)$ to these cofiber sequences, we claim that all the connecting homomorphisms induced by the dashed arrows in the telescope tower are trivial.  This is because the connecting homomorphisms are maps of $(E_n)_*$-modules from a $v_k$-torsion $(E_n)_*$-module to a $v_k$-torsion free $(E_n)_*$-module ($0 \leqslant k < n$). An inductive computation shows that $\pi_*(E_n \otimes S^0/I_n^\infty) \cong (E_n)_*/I_n^\infty$.  Thus,  $E_n/I_n^\infty \simeq E_n \otimes S^0/I_n^\infty$.
  
Since $M_n(v_k^{-1}S^0/I_k^{\infty}) \simeq \ast$ for $k < n$, after applying the $n$-th monochromatic layer we have $M_n(S^0/I_k^\infty) = \Sigma^{-1}M_n(S^0/I_{k+1}^\infty )$ for $k < n$.  Hence, we have a weak equivalence 
$$t \colon \Sigma^{-n}M_n(S^0/I_n^\infty) \to M_nS^0.$$

After smashing with $E_n$, we produce a $\mathbb{G}_n$-equivariant map 
\[
t \otimes \text{id}_{E_n} \colon \Sigma^{-n}M_n(S^0/I_n^\infty) \otimes E_n \to M_nS^0 \otimes E_n,
\]
where the action is entirely on $E_n$. 
Because $L_n$ is a smashing localization, we have $ \Sigma^{-n}M_n(S^0/I_n^\infty) \otimes E_n \simeq \Sigma^{-n}M_n(S^0/I_n^\infty  \otimes E_n) $.  The latter is $\Sigma^{-n}M_n(E_n/I_n^\infty) = \Sigma^{-n}E_n/I_n^\infty$. Therefore, $t \otimes \text{id}_{E_n}$ gives the desired $\mathbb{G}_n$-equivariant equivalence
  \[
\Sigma^{-n}E_n/I_n^{\infty}\xr{\simeq} M_nE_n.\qedhere
  \]

  \end{proof}

\subsection{\texorpdfstring{$C_2$}{C2}-equivariant homotopy}
Our first proof of the Gross--Hopkins dual of $E_n^{hC_2}$ will rely on techniques from $C_2$-equivariant homotopy theory, and so we begin by giving some background on what we will require. More general references include \cite[App.~A]{hhr}, \cite{hukriz}, \cite[Sec.~2]{hm_tmf}, \cite{green_may_tate}, and \cite{greenlees_four}. 

We will work in the category of genuine $C_2$-spectra. It follows that for a $C_2$-spectrum $X$ we can define homotopy groups $\pi^{C_2}_{x+y\sigma}X$ indexed on $RO(C_2) = \{ x + y\sigma \mid x,y \in \Z\}$, where $\sigma$ denotes the sign representation of $C_2$. We will use Hu and Kriz's convention of writing $\pi^{C_2}_{\star}X$ for $RO(C_2)$-graded homotopy groups, and $\pi^{C_2}_*X$ when we grade only over the integers. We denote by $\pi_*^eX$ the underlying homotopy groups of $X$ if there is any potential confusion. We also recall that the equivariant and non-equivariant homotopy groups can be combined into a Mackey functor, which we write as $\underline{\pi}_*X$. 

Following Hill--Meier \cite[Def.~3.1]{hm_tmf} we say that a $C_2$-spectrum $X$ is strongly even if $\underline{\pi}_{k\rho-1}X = 0$, and $\underline{\pi}_{k\rho}X$ is a constant Mackey functor for all $k \in \Z$, where $\rho = 1 + \sigma$ denotes the real regular representation. As in \cite{greenlees_meier} if $X$ is strongly even, and $e \in \pi^e_{2k}X$, we let $\bar{e}$ denote its counterpart in $\pi^{C_2}_{k\rho}X$, i.e., the preimage of $e$ under the isomorphism
\[
\pi_{k\rho}^{C_2}X \xr{\cong} \pi_{k\rho}^eX \cong \pi_{2k}^eX. 
\]  

\begin{ex}\label{ex:stronglyeven}
  Let $MU_\mathbb{R}$ denote the real bordism spectrum considered by Araki and Hu--Kriz \cite{hukriz}. This is a strongly even $C_2$-spectrum \cite[App.~A]{hm_tmf}. Similarly, as a $C_2$-spectrum, $E_n$ is strongly even, see the proof of Theorem 6.7 of \cite{hahn_shi}. The key input for this is the existence of a $C_2$-equivariant map $MU_\mathbb{R} \to E_n$, see \cite[Thm.~1.2]{hahn_shi}. 
\end{ex}
The following lemma of Greenlees \cite[Lem.~1.2]{greenlees_four} will prove useful. For this, we say that a $C_2$-spectrum $X$ is underlying even if $\pi^e_{2k-1}X = 0$ for all $k \in \Z$. 
\begin{lem}[Greenlees]\label{lem:greenstronglyeven}
  A $C_2$-spectrum $X$ is strongly even if and only if it is underlying even and $\pi^{C_2}_{\ast \rho -i}X = 0$ for $i = 1,2,3$. 
\end{lem}

Consider the contractible free $C_2$-space $EC_2$ and the join $\tilde{E}C_2 \simeq S^0 \ast EC_2$, which fit together in a cofiber sequence
\[
EC_{2+} \to S^0 \to \tilde{E}C_2. 
\] 
We can then construct the following diagram of cofiber sequences for any $C_2$-spectrum $X$
\[
\setlength\mathsurround{0pt}
\begin{tikzcd}
X_h \arrow{r} \arrow{d}{\simeq} & X \arrow{r} \arrow{d} &  X^{\Phi} \arrow{d} \\
X_h \arrow{r} & X^h \arrow{r} & X^t. 
\end{tikzcd}
\]
where\footnote{In \cite{greenlees_four} Greenlees calls this the Scandinavian notation.}
\[
\begin{split}
X^h &= F(EC_{2+},X) \\
X_h &= EC_{2+} \otimes X \\
X^t &= F(EC_{2+},X) \otimes \tilde{E}C_2\\
X^{\Phi} &= X \otimes \tilde{E}C_{2}.
\end{split}
\]
Here the notation $X^{\Phi}$ is justified because $C_2$ has prime order, and hence the $C_2$-fixed points of $X^{\Phi}$ compute the geometric fixed points.\footnote{In general, the geometric fixed points of a $G$-spectrum $X$ are defined in the following way. Let $\tilde E\cal{F}$ be the $G$-CW complex defined as the cofiber of the map from $E\cal{F}_+ \to S^0$, where $E\cal{F}$ is the universal space associated to the family of proper subgroups of $G$. Then, the geometric fixed points of $X$ are the $G$-fixed points of $\tilde E\cal{F} \otimes X$, see \cite[Sec.~2.5.2]{hhr}. In the particular case of $C_2$, we have $\tilde E\cal{F} \simeq \tilde EC_2$, and the claim follows.} We have also used that the map $EC_{2+} \otimes X \to EC_{2+} \otimes F(EC_{2+},X)$ induced by $EC_{2+} \to S^0$ is always an equivalence of $C_2$-spectra, see \cite[Proposition I.1.2]{green_may_tate}. 

 Passing to $C_2$-fixed points we obtain the following diagram of cofiber sequences
\begin{equation}\label{eq:tate}
   \setlength\mathsurround{0pt}
\begin{tikzcd}
  X_{hC_2} \arrow{r} \arrow{d}[swap]{\simeq} & X^{C_2} \arrow{r} \arrow{d} & X^{\Phi C_2} \arrow{d} \\
  X_{hC_2} \arrow{r} & X^{hC_2} \arrow{r} & X^{tC_2}
\end{tikzcd}
\end{equation} 
where the bottom cofiber sequence is the norm sequence. 
\begin{defn}
  A $C_2$-spectrum $X$ is cofree if the map $X \to X^h = F(EC_{2+},X)$ is an equivalence. 
\end{defn}
As explained in \cite[Sec.~2.2.1]{hm_tmf} if $X$ is simply a spectrum with a $C_2$-action (i.e., an object of the functor category $\Fun(BC_2,\Sp)$), then $F(EC_{2+},X)$ is a genuine equivariant spectrum. In fact, by \cite[Thm.~2.4]{hm_tmf} if $X$ is a commutative ring spectrum with a $C_2$-action via commutative ring maps, then $F(EC_{2+},X)$ is an equivariant commutative ring spectrum. The conditions on $X$ are satisfied when $X = E_n$ by the Goerss--Hopkins--Miller theorem \cite{goerss_hopkins}. We use this to view $E_n$ as a genuine $C_2$-equivariant spectrum.
\begin{defn}
  The genuine $C_2$-equivariant Lubin--Tate spectrum $E_n$ is defined to be to cofree $C_2$-spectrum $F(EC_{2+},E_n)$. 
\end{defn}
By construction, this has the property that 
\begin{equation}\label{cor:fixvshfp}
E_n^{C_2} \cong E_n^{hC_2}. 
\end{equation}

Associated to the norm sequence are spectral sequences
\[
\begin{split}
H_s(C_2,\pi_tX) &\implies \pi_{t+s}X_{hC_2} \\
H^s(C_2,\pi_tX) &\implies \pi_{t-s}X^{hC_2} \\
\hat{H}^s(C_2,\pi_tX) & \implies \pi_{t-s} X^{tC_2}.
\end{split}
\]
In fact, these can be made into $RO(C_2)$-graded spectral sequences computing $\pi^{C_2}_{\star}X_h$, $\pi^{C_2}_{\star}X^h$ and $\pi^{C_2}_{\star}X^t$, respectively. By looking only at the integer degrees of these spectral sequence, we recover the standard spectral sequences. Although this seems like more work, it may be the case that the differentials are easier to describe, or to compute, in the $RO(C_2)$-grading. This is especially true in light of the large work on computation in $C_2$-equivariant homotopy theory, pioneered by Hu and Kriz \cite{hukriz}, and later by Hill, Hopkins, and Ravenel \cite{hhr}. 

\section{The Picard group of \texorpdfstring{$E_n^{hC_2}$}{EnhC2}}
In this section we use the descent techniques of Mathew and Stojanoska to compute the Picard group $\Pic(E_n^{hC_2})$. 

\subsection{The \texorpdfstring{$RO(C_2)$}{RO(C2)}-graded homotopy fixed point spectral sequence for \texorpdfstring{$E_n$}{En}}
Once again, we let $E_n$ denote Morava $E$-theory at the prime $2$. Hahn and Shi have shown \cite[Thm.~1.2]{hahn_shi} that there is a equivariant map from $MU \R \to E_n$, which induces a map of spectral sequences 
\[
C_2\text{-}\HFPSS(MU\mathbb{R}) \to C_2\text{-}\HFPSS(E_n).
\]
Using this Hahn and Shi \cite[Thm.~6.2]{hahn_shi} were able to fully compute the $RO(C_2)$-graded HFPSS for $E_n$. 
\begin{thm}[Hahn--Shi]\label{thm:additivess}
  \begin{enumerate}
    \item The $\E_2$-term of the $RO(C_2)$-graded homotopy fixed point spectral sequence for $E_n$ is
  \[
\E_2^{s, t} = \mathbb{W}(\mathbb{F}_{2^n})[\![\bar{u}_1, \bar{u}_2, \ldots, \bar{u}_{n-1}]\!][\bar{u}^{\pm 1}] \otimes \mathbb{Z}[u_{2\sigma}^{\pm 1}, a_\sigma]/(2a_\sigma).
\]
Here the classes have bidegrees $|\bu_i| = (0,0)$, $|\bar{u}| = (\rho,0)$, $|u_{2\sigma}| = (2-2\sigma,0)$ and $|a_{\sigma}| = (-\sigma,1)$. 
\item The classes $\bar{u}_1$, $\ldots$, $\bar{u}_{n-1}$, $\bar{u}^{\pm 1}$, and $a_\sigma$ are permanent cycles.  All the differentials in the spectral sequence are determined by the differentials 
\begin{eqnarray*}
d_{2^{k+1} -1} (u_{2\sigma}^{2^{k-1}}) &=&  \bar{u}_k\bar{u}^{2^k-1}a_\sigma^{2^{k+1}-1}, \, \, \, 1 \leq k \leq n-1, \\ 
d_{2^{n+1}-1}(u_{2\sigma}^{2^{n-1}})&=& \bar{u}^{2^n-1}a_\sigma^{2^{n+1}-1}, \, \, \, k = n, \
\end{eqnarray*}
and multiplicative structures. 
  \end{enumerate}

\end{thm}
\begin{rem}
  Working non-equivariantly it is not hard to compute the $\E_2$-term of the homotopy fixed point spectral sequence computing $E_n^{hC_2}$, however determining the differentials directly seems difficult. The use of $RO(C_2)$-grading is helpful in determining these, as the differentials in the spectral sequence for $MU_\mathbb{R}$ are fully understood by work of Hu and Kriz \cite{hukriz}.
\end{rem}
\begin{rem}\label{rem:diff_periodicity}
  We observe that the $d_{2^{k+1}-1}$-st differential is $2^{k+2}$-periodic, with periodicity generator $u_{2\sigma}^{2^{k}}\bu^{2^{k+1}}$.
\end{rem}
\begin{cor}
The homotopy groups $\pi_*E_n^{hC_2}$ are $2^{n+2}$-periodic. 
\end{cor}
\begin{proof}
  The invertible class $\bu^{2^{n+1}} u_{2\sigma}^{2^n}$ survives the homotopy fixed point spectral sequence. 
\end{proof}

\begin{figure}[hbt!]
\begin{center}
  \scalebox{.8}{\includegraphics{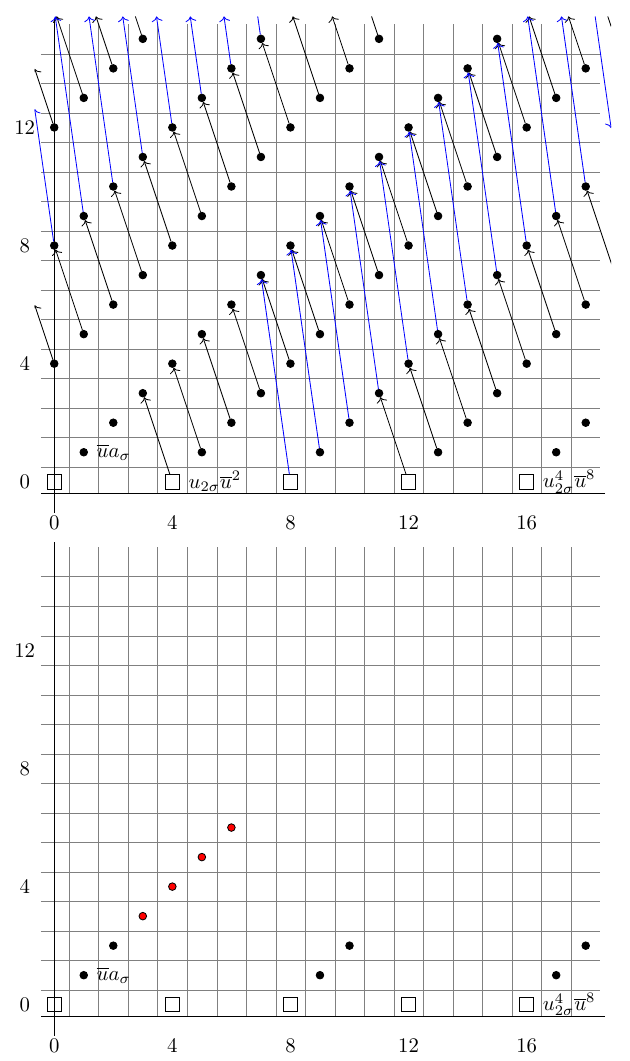}}
  \caption{The homotopy fixed point spectral sequence for $E_2^{hC_2}$. Here the squares denote copies of $\mathbb{W}(\F_4)[\![u_1]\!]$, black circles denote $\F_4[\![u_1]\!]$, and red circles $\F_4$. The top image shows the $\E_2$-page along with all $d_3$ and $d_7$ differentials. The bottom image shows the $E_{\infty}$-page, which is 16-periodic on the class $u_{2\sigma}^4\bu^8$.}\label{fig:hfpssadditive}
  \end{center}
\end{figure}
We demonstrate this spectral sequence for $n = 2$ in \Cref{fig:hfpssadditive}. The case of $n=3$ can be found in \cite{hahn_shi}. We note that in both these cases we have $\pi_kE_n^{hC_2} = 0$ for $k \equiv -3,-2,-1 \pmod{2^{n+2}}$. This holds for arbitrary $n$, and is in fact the only such gap in the spectral sequence. 
\begin{prop}\label{prop:recognition}
    We have $\pi_kE_n^{hC_2} = \pi_{k+1}E_n^{hC_2} = \pi_{k+2}E_n^{hC_2} = 0$ if and only if $k \equiv -3 \pmod{2^{n+2}}$. 
\end{prop}
\begin{proof}
  By \Cref{ex:stronglyeven} $E_n$ is strongly even as a $C_2$-spectrum. We can now apply \Cref{lem:greenstronglyeven} and \eqref{cor:fixvshfp} to deduce that $\pi_{k}^{C_2}E_n \cong \pi_kE_n^{hC_2} = 0$ for $k = -3,-2,-1$. It remains to show that, modulo periodicity, this is the only such gap of three zero terms. Note that $\pi_kE_n^{hC_2} \ne 0$ for $k \equiv 0 \pmod{4}$ because the classes $u_{2\sigma}^k\bu^{2k}$ in bidegree $(4k,0)$ are permanent cycles in the spectral sequence, so this is the longest such possible gap of zeros. 

  We will show that there is no such other gap by exhibiting permanent cycles in the homotopy fixed point spectral sequence. We make the following observations, which are the essential point to the proof:\\

  \noindent Observation (1): The classes $\bu^i a_{\sigma}^i$ are the targets of between 0 and $n$ differentials, depending on the filtration degree (the class in filtration degree $2^{k+1}-1$ is the target of $k$-differentials), with differentials determined by 
  \[
d_{2^{k+1}-1}(u_{2\sigma}^{2^{k-1}}\bu^{2^k}) = \bu_k \bu^{2^{k+1}-1}a_{\sigma}^{2^{k+1}-1}, \quad 1 \le k \le n
  \]
  where $\bu_n = 1$. \\

  \noindent Observation (2): The classes $\bu^{2^n+i}a_{\sigma}^{i}u_{2\sigma}^{2^{n-1}}$ are the targets of between 0 and $n-1$ differentials, depending on the filtration degree (the class in filtration degree $2^{k+2}-1$ is the target of $k$-differentials), with differentials determined by 
\[
d_{2^{k+1}-1}(u_{2\sigma}^{2^{n-1}+k}\bu^{2^n+2k}) = \bu_k \bu_{2\sigma}^{2^{n-1}+k}\bu^{2^n+2^{k+1}-1}a_{\sigma}^{2^{k+1}-1}, \quad 1 \le k \le n-1. 
\]
 \\
 
\noindent Observation (3): In general, for $1 \le j \le n$, the classes $e^i_{j,n} = (\bu a_{\sigma})^i\bar{u}^{(2^{j-1}-1)\cdot2^{n+2-j}}u_{2\sigma}^{(2^{j-1}-1)\cdot 2^{n+1-j}}$ are the targets of between 0 and $n - j$ differentials, depending on the filtration degree (the class in filtration degree $2^{k+j+1}-1$ is the target of $k$-differentials), with differentials determined by 
  \[
  \begin{split}
d_{2^{k+1}-1}(\bar{u}^{(2^{j-1}-1)\cdot2^{n+2-j}+k}&u_{2\sigma}^{(2^{j-1}-1)\cdot 2^{n+1-j}+2k}) =\\ &\bu_k \bar{u}^{(2^{j-1}-1)\cdot2^{n+2-j}+2^{k+1}-1}u_{2\sigma}^{(2^{j-1}-1)\cdot 2^{n+1-j}}a_{\sigma}^{2^{k+1}-1}, \quad 1 \le k \le n-j.
\end{split}
  \]
  We now break the proof down into a number of steps. \\

\noindent (1)   First, we observe that every class in positive filtration in the $\cE_2$-page of the spectral sequence is of the form 
\[
\mathbb{F}_{2^n}[\![\bar{u}_1, \bar{u}_2, \ldots, \bar{u}_{n-1}]\!][\bar{u}^{\pm 1}]\bu^{\alpha}u_{2\sigma}^{\beta}a_{\sigma}^{\gamma}. 
\]
 It follows from \Cref{thm:additivess} that the classes $\bu^{k}a_{\sigma}^{k}$ for $1 \le k \le 2^{n+1}-2$ all survive to give permanent cycles in $\pi_kE_n^{hC_2}$. Note that we do not claim that these classes are not involved in differentials. Indeed this is false, and from Observation (1) we see that these classes are the target of between 0 and $n-1$ differentials, depending on the filtration degree. For example, consider the topmost class $\bu^{2^{n+1}-2}a_{\sigma}^{2^{n+2}-2}$. This is the target of $n-1$ non-trivial differentials, where the $d_{2^{k+1}-1}$-differential (for $1 \le k \le n-1$) quotients out by the principal ideal $(\bar{u}_k)$. It follows that a single copy of $\F_{2^n}$ survives the spectral sequence at this point. (In \Cref{fig:hfpssadditive} one can see this when $n = 2$; the red class in bidegree (6,6) is precisely the claimed copy of $\F_4$.) We deduce that there can be no such gap for $1 \le k \le 2^{n+1}$.  \\

  \noindent (2) Similarly, we claim that the classes $\bu^{2^n+i}a_{\sigma}^{i}u_{2\sigma}^{2^{n-1}}$ for $1 \le i \le 2^{n}-2$ survive and give permanent cycles in $\pi_kE_n^{hC_2}$ for $2^{n+1}+1 \le k \le 3\cdot2^n-2$. Once again, we do not claim that these classes are not involved in differentials. Using Observation (2), we see that these classes are the target of at most $n-2$ non-trivial differentials, where each differential quotients out by some $(\bu_k)$.  In this case, the classes $\bu^{2^n+i}a_{\sigma}^{i}u_{2\sigma}^{2^{n-1}}$ also support a non-trivial differential, determined by
  \[
d_{2^{n+1}-1}(\bu^{2^n+i}u_{2\sigma}^{2^{n-1}}a_{\sigma}^i) = \bu^{2^{n+1}+i-1}a_{\sigma}^{2^{n+1}+i-1}. 
  \] 
  However, the targets of these differentials are always copies of $\mathbb{F}_{2^n}$, because they have already supported $n-1$ non-trivial differentials (indeed, they are exactly the classes considered in Step (1)). It follows that these differentials always have non-trivial kernel, and so there are non-trivial classes in these bidegrees as claimed. Because we always have non-zero classes for $k \equiv 0 \pmod{4}$ we see that there is no gap for $0 \le k \le 3\cdot2^{n}$. \\

  \noindent (3) In general, for $1 \le j \le n$, the classes $e^i_{j,n}=(\bu a_{\sigma})^i\bar{u}^{(2^{j-1}-1)\cdot2^{n+2-j}}u_{2\sigma}^{(2^{j-1}-1)\cdot 2^{n+1-j}}$ for $1 \le i \le 2^{n+2-j}-2$ survive, and contribute non-zero classes in $\pi_kE_n^{hC_2}$ for \[(2^{j-1}-1)\cdot 2^{n+3-j} +1  \le k \le (2^j-1)\cdot 2^{n+2-j}-2.
  \]
  One can see this by arguing similarly to as done previously, which considered the cases where $j = 1$ and $2$. Depending on the filtration degree, we see from Observation (3) that these classes are the target of between $0$ and $n-j$ non-trivial differentials.

   For $j>1$, we claim that the $e^{i}_{j,n}$ are also the source of a $d_{2^{n+3-j}-1}$-differential. To see this, observe that 
  \[
  (2^{j-1}-1)\cdot 2^{n+1-j} = (2^{j-2}-1)\cdot 2^{n+2-j}+2^{n+1-j}
\]
Using the periodicity of the differentials (\Cref{rem:diff_periodicity}) it then follows there is a differential
  \[
d_{2^{n+3-j}-1}(e^i_{j,n}) = \bu_{n+2-j}(\bu a_{\sigma})^i u_{2\sigma}^{(2^{j-2}-1)\cdot 2^{n+1-j}}a_{\sigma}^{2^{n+3-j}-1}\bu^{(2^{j-2}+1)\cdot 2^{n+2-j}}
  \] 
  where we again make the convention that $\bu_n = 1$. Observe that the target of this differential is simply $\bu_{n+2-j}e^i_{j-1,n}$. Moreover, using Observation (3), we see that this target supports at most $n - j +1$ differentials, depending on the filtration degree. In fact, comparing the filtration degrees of the source and target, we see that the target always supports one more differential than the source; in particular, the $d_{2^{n+3-j}-1}$-differential above has non-trivial kernel, and hence gives a permanent cycle (it is the last non-trivial differential involving this class).   

  Inductively, we see there is no gap in $\pi_kE_n^{hC_2}$ for $0 \le i \le (2^j-1) \cdot 2^{n+2-j}$. In particular, taking $n =j$, we see there is no gap for $0 \le i \le (2^n-1) \cdot 4$, i.e, from $0 \le i \le 2^{n+2}-4$.  By the $2^{n+2}$-periodicity of $\pi_*E_n^{hC_2}$ we are done. 
\end{proof}
\begin{rem}
  It may be useful to give the following visual guide to identifying these permanent cycles. For $j = 1$ we have a line of slope 1 and length $2^{n+1}-2$ beginning from position $(1,1)$ in the spectral sequence. For $j = 2$ we have a line of slope 1 and length $2^{n}-2$ beginning in position $(2^{n+1}+1,1)$. For $j = 3$ we have a line of slope 1, and length $2^{n-1}-2$ beginning in position $(3 \cdot 2^n +1,1)$, and so on. See \Cref{fig:hfpssadditive} for the case $n = 2$ and \cite[Fig.~7]{hahn_shi} for the case $n = 3$. 
\end{rem}
This gives rise to the following important corollary. 
\begin{cor}\label{cor:recognition}
  Suppose that $X \simeq \Sigma^\ell E_n^{hC_2}$ for some integer $\ell$ (which is only uniquely determined modulo $2^{n+2}$).  If $\pi_jX = \pi_{j+1}X = \pi_{j+2}X = 0$, then $\ell = j+3$. 
\end{cor}
\subsection{The Picard spectral sequence}
We briefly review the techniques introduced in \cite{ms_pic} which we will use to compute the Picard group of $E_n^{hC_2}$-modules, which we denote $\Pic(E_n^{hC_2})$. We recall that for any $E_{\infty}$-ring spectrum $R$, there exists a connective spectrum $\pics(R)$ with the property that 
\[
\pi_i\pics(R) \cong 
\begin{cases}
  \Pic(R) & i = 0 \\
  \pi_0(R)^{\times} & i = 1 \\
  \pi_{i-1}(R) & i \ge 2. 
\end{cases}
\]
There is a faithful $C_2$-Galois extension $E_n^{hC_2} \to E_n$, see \cite[Prop.~3.6]{hms_pic}, and the techniques of \cite{ms_pic} then apply to show that there is an equivalence of connective spectra
\[
\pics(E_n^{hC_2}) \simeq \tau_{\ge 0}\pics(E_n)^{hC_2},
\]
where $\tau_{\ge 0}$ denotes the connective cover. In particular, there is a Picard spectral sequence 
\[
\E_{2,\times}^{s,t} \cong H^s(C_2,\pi_t\pics(E_n)) \implies \pi_{t-s}\pics(E_n)^{hC_2}
\]
for $s,t \ge 0$, whose abutment for $t = s$ is the Picard group $\Pic(E_n^{hC_2})$. Note that for $t \ge 2$ the $E_2$-term of the Picard spectral sequence is just a shift of the ordinary integer graded homotopy fixed point spectral sequence. We call this the stable range of the spectral sequence. In fact, even more is true; by the comparison tool of Mathew and Stojanoska \cite[5.2.4]{ms_pic}, whenever $2 \le r \le t-1$ we have an equality of differentials $d_{r,\times}^{s,t} = d_{r}^{s,t-1}$, where $d_{r,\times}$ denotes the $r$-th differential in the Picard spectral sequence, and $d_r$ the $r$-th differential in the ordinary homotopy fixed point spectral sequence. 

Since we are interested in terms contributing to $\pi_0\pics(E_n)^{hC_2}$, we should look at classes in the $(-1)$-stem of the additive spectral sequence. By degree reasons these must have the form
\[
 \F_{2^n}[\![u_1,\ldots,u_{n-1}]\!]\bu^{2\ell-1}u_{2\sigma}^{-\ell}a_{\sigma}^{4\ell-1}. 
\]
in bidegree $(-1,4\ell-1	)$, where $\ell \ge 1$. The reader should compare the following to the proof of \cite[Lemma.~6.9]{hahn_shi}.
\begin{prop}
	The only classes in the stable range that can contribute to $\pi_0(\pics(E_n^{hC_2}))$ are in filtration degree $2^k-1$ for $1 \le k \le n$. 
\end{prop}
\begin{proof}

\noindent (1) The $d_3$-differentials are generated by
\[
d_3(u_{2\sigma}) = \bu_1\bu a_{\sigma}^{3}. 
\]
In particular, if $\ell \equiv 1 \pmod 2$, then 
\[
d_3(\bu^{2\ell-1}u_{2\sigma}^{-\ell}a_{\sigma}^{4\ell-1}) = \bu_1 \bu^{2\ell} u_{2\sigma}^{-\ell-1}a_{\sigma}^{4\ell+2}
\]
and hence all the classes of the form 
\[
\bu^{2\ell-1}u_{2\sigma}^{-\ell}a_{\sigma}^{4\ell-1}, \quad \ell \equiv 1 \pmod 2
\]
die on the $\E_3$-page of the additive spectral sequence. By the comparison tool \cite[5.2.4]{ms_pic} we can import these $d_3$-differentials whenever the classes lie in filtration degree greater than 4, i.e., when $4\ell -1 \ge 4$. We conclude we can import all these $d_3$-differentials except the differential originating in filtration degree 3 (corresponding to $\ell = 1$). 

The classes 
\[
\bu^{2\ell-1}u_{2\sigma}^{-\ell}a_{\sigma}^{4\ell-1}, \quad \ell \equiv 0 \pmod 2
\]
are the targets of the $d_3$-differential
\[
d_3(\bu^{2\ell-2}u_{2\sigma}^{-\ell+1}a_{\sigma}^{4\ell-4})= \bu_1\bu^{2\ell-1}u_{2\sigma}^{-\ell}a_{\sigma}^{4\ell-1}
\]

By the comparison tool we can import a $d_3$ differential whenever the source has $t \ge 4$. Here the source has degree $4\ell-4$, and so we can import these whenever $4\ell-4 \ge 4$, i.e. $\ell \ge 2$. Since $\ell \ge 1$ and $\ell \equiv 0 \pmod 2$, we can import all these differentials. The differential quotients by the principal ideal $(\bu_1)$, and hence the remaining classes in the stable range that can contribute have the form
\[
\F_{2^n}[\![\bu_{2},\ldots,\bu_{n-1}]\!]\bu^{2\ell-1}u_{2\sigma}^{-\ell}a_{\sigma}^{4\ell-1}. 
\]
with $\ell \equiv 0 \pmod 2$ and $\ell \ge 1$. \\

\noindent (2) The $d_7$-differentials are generated by
\[
d_7(u_{2\sigma}^2) = \bu_2\bu^3a_{\sigma}^{7}. 
\]
In particular, if $\ell \equiv 2 \pmod 4$, then 
\[
d_7(\bu^{2\ell-1}u_{2\sigma}^{-\ell}a_{\sigma}^{4\ell-1}) = \bu_2 \bu^{2\ell+2} u_{2\sigma}^{-\ell-2}a_{\sigma}^{4\ell+6}. 
\]
and hence all the classes of the form 
\[
\bu^{2\ell-1}u_{2\sigma}^{-\ell}a_{\sigma}^{4\ell-1}, \quad \ell \equiv 2 \pmod 4
\]
die on the $\E_7$-page of the additive spectral sequence. By the comparison tool we can import these $d_7$-differentials whenever the classes lie in filtration degree greater than 8, i.e., when $4\ell -1 \ge 8$. We conclude we can import all these $d_7$-differentials except the differential originating in filtration degree $7$ (corresponding to $\ell = 2$). 

The remaining classes 
\[
\bu^{2\ell-1}u_{2\sigma}^{-\ell}a_{\sigma}^{4\ell-1}, \quad \ell  \equiv 0 \pmod 4
\] are the targets of the $d_7$-differential
\[
d_7(\bu^{2\ell-4}u_{2\sigma}^{-\ell+2}a_{\sigma}^{4\ell-8}) = \bu_2\bu^{2\ell-1}u_{2\sigma}^{-\ell}a_{\sigma}^{4\ell-1}
\]

 By the comparison tool we can import a $d_7$ differential whenever the source has $t \ge 8$. Here the source has degree $4\ell-8$, and so we can import these whenever $4\ell-8 \ge 8$, i.e. $\ell \ge 4$. Since $\ell \ge 1$ and $\ell \equiv 0 \pmod 4$, we can import all these $d_7$-differentials. The differential quotients by the principal ideal $(\bu_2)$, and hence the remaining classes have the form
\[
\F_{2^n}[\![\bu_3,\ldots,\bu_n]\!]\bu^{2\ell-1}u_{2\sigma}^{-\ell}a_{\sigma}^{4\ell-1}. 
\]
with $\ell \equiv 0 \pmod 4$. \\

\noindent (3) For $0 < k < n$, the $d_{2^{k+1}-1}$-differentials are generated by
\[
d_{2^{k+1}-1}(u_{2\sigma}^{2^{k-1}}) = \bu_k\bu^{2^{k}-1}a_{\sigma}^{2^{k+1}-1}. 
\]
In particular, if $\ell \equiv 2^{k-1} \pmod {2^k}$, then 
\[
d_{2^{k+1}-1}(\bu^{2\ell-1}u_{2\sigma}^{-\ell}a_{\sigma}^{4\ell-1}) = \bu_k\bu^{2\ell+2^{k}-2}u_{2\sigma}^{-\ell-2^{k-1}}a_{\sigma}^{4\ell+2^{k+1}-2}
\]
and hence the classes 
\[\bu^{2\ell-1}u_{2\sigma}^{-\ell}a_{\sigma}^{4\ell-1}, \quad \ell \equiv 2^{k-1} \pmod {2^k}
\] die on the $\E_{2^{k+1}-1}$-page of the additive spectral sequence. By the comparison tool we can import these $d_{2^{k+1}-1}$-differentials whenever the classes lie in filtration degree greater than $2^{k+1}$, i.e., when $4\ell -1 \ge 2^{k+1}$. We conclude we can import all these $d_{2^{k+1}-1}$-differentials except the differential originating in filtration degree $2^{k-1}$ (corresponding to $\ell = 2^{k-1}$). 

The classes of the form
\[
\bu^{2\ell-1}u_{2\sigma}^{-\ell}a_{\sigma}^{4\ell-1}, \quad \ell \equiv 0 \pmod {2^k}.
\] 
are the targets of a $d_{2^{k+1}-1}$-differential
\[
d_{2^{k+1}-1}(\bu^{2\ell -2^k}\bu_{2\sigma}^{\ell+2^{k-1}}a_{\sigma}^{4\ell - 2^{k+1}}) = \bu_k\bu^{2\ell-1}u_{2\sigma}^{-\ell}a_{\sigma}^{4\ell-1}
\] By the comparison tool we can import a $d_{2^{k+1}-1}$-differential whenever the source has $t \ge 2^{k+1}$. Here the source has degree $4\ell-2^{k+1}$, and so we can import these whenever $4\ell-2^{k+1} \ge 2^{k+1}$, i.e. $\ell \ge 2^{k}$. Since $\ell \ge 1$ and $\ell \equiv 0 \pmod {2^k}$, we can import all these $d_{2^{k+1}-1}$-differentials. The differential quotients by the principal ideal $(\bu_k)$, and hence the remaining classes have the form
\[
\F_{2^n}[\![u_{k+1},\ldots,u_n]\!]\bu^{2\ell-1}u_{2\sigma}^{-\ell}a_{\sigma}^{4\ell-1}. 
\]
with $\ell \equiv 0 \pmod {2^k}$. \\
\noindent (4) The $d_{2^{n+1}-1}$-differentials are generated by
\[
d_{2^{n+1}-1}(u_{2\sigma}^{2^{n-1}}) = \bu^{2^{n}-1}a_{\sigma}^{2^{n+1}-1}. 
\]
In particular, if $\ell \equiv 2^{n-1} \pmod {2^n}$, then 
\[
d_{2^{n+1}-1}(\bu^{2\ell-1}u_{2\sigma}^{-\ell}a_{\sigma}^{4\ell-1}) = \bu^{2\ell+2^{n}-2}u_{2\sigma}^{-\ell-2^{n-1}}a_{\sigma}^{4\ell+2^{n+1}-2}
\]
and hence the classes 
\[\bu^{2\ell-1}u_{2\sigma}^{-\ell}a_{\sigma}^{4\ell-1}, \quad \ell \equiv 2^{n-1} \pmod {p^n}\]
 die on the $\E_{2^{n+1}-1}$-page of the additive spectral sequence. By the comparison tool we can import these $d_{2^{n+1}-1}$-differentials whenever the classes lie in filtration degree greater than $2^{n+1}$, i.e., when $4\ell -1 \ge 2^{n+1}$. We conclude we can import all these $d_{2^{n+1}-1}$-differentials except the differential originating in filtration degree $2^n$ (corresponding to $\ell = 2^{n-1}$). 

The classes 
\[
\bu^{2\ell-1}u_{2\sigma}^{-\ell}a_{\sigma}^{4\ell-1}, \quad \ell \equiv 0 \pmod {2^n}
\]are the targets of a $d_{2^{n+1}-1}$-differential,
\[
d_{2^{n+1}-1}(\bu^{2\ell -2^n}\bu_{2\sigma}^{\ell+2^{n-1}}a_{\sigma}^{4\ell - 2^{n+1}}) = \bu^{2\ell-1}u_{2\sigma}^{-\ell}a_{\sigma}^{4\ell-1}.
\] 
By the comparison tool we can import a $d_{2^{n+1}-1}$ differential whenever the source has $t \ge 2^{n+1}$. Here the source has degree $4\ell-2^{n+1}$, and so we can import these whenever $4\ell-2^{n+1} \ge 2^{n+1}$, i.e. $\ell \ge 2^{n}$. Since $\ell \ge 1$ and $\ell \equiv 0 \pmod {2^n}$, we can import all these $d_{2^{n+1}-1}$-differentials. Using (3) above the targets are just copies of $\F_{2^n}$ and hence die after these differentials. 

It follows that the only possible contributions to the Picard spectral sequence in the stable range are those classes in filtration degree $2^{k}-1$ for $1 \le k \le n$ as claimed. 
\end{proof}
We now determine the differentials for these classes in the Picard spectral sequence. 
\begin{prop}\label{prop:stablediffpic}
	In the zero stem, and the stable range of the Picard spectral sequence for $\pics(E_n)^{hC_2}$, there is a group of order at most $2$ in filtration degree $2^{k+1}-1$ for $1 \le k \le n$. 
\end{prop}
\begin{proof}
	We recall from \cite[Thm.~6.1.1]{ms_pic} that the first differential outside the stable range in the Picard spectral sequence is given by
	\begin{equation}\label{eq:picdiff}
d^{\Pic}_r(x) = d_r(x) + x^2, \quad x \in \E^{r,r}_{r,\Pic}
	\end{equation}
  where we abuse notation and denote by $x$ the same class in the Picard and additive homotopy fixed point spectral sequences. 

	Let $f = f(u_k,\ldots,u_{n-1}) \in \F_{2^n}[\![u_k,\ldots,u_{n-1}]\!]$. For $1 \le k \le n-1$ the additive differential that we could not import is given by
	\[
d_{2^{k+1}-1}(f\bu^{2^k-1}u_{2\sigma}^{-2^{k-1}}a_{\sigma}^{2^{k+1}-1}) = f\bu_k\bu^{2^{k+1}-2}u_{2\sigma}^{-2^k}a_{\sigma}^{2^{k+2}-2}. 
	\]
By \eqref{eq:picdiff} the corresponding differential in the Picard spectral sequence is given by
\[
\begin{split}
d^{\Pic}_{2^{k+1}-1}(f\bu^{2^k-1}u_{2\sigma}^{-2^{k-1}}a_{\sigma}^{2^{k+1}-1}) &=  f\bu_k\bu^{2^{k+1}-2}u_{2\sigma}^{-2^k}a_{\sigma}^{2^{k+2}-2} + f^2\bu^{2^{k+1}-2}u_{2\sigma}^{-2^{k}}a_{\sigma}^{2^{k+2}-2} \\
&= (\bu_kf+f^2)\bu^{2^{k+1}-2}u_{2\sigma}^{-2^k}a_{\sigma}^{2^{k+2}-2}.
\end{split}
\]
This is zero whenever $\bu_kf+f^2 = 0$. Since $\bu_kf+f^2 = f(\bu_kf+1)$, we see that there are precisely two solutions, namely $f = 0$ and $f = \bu_k$, and hence the kernel generates a group of order 2. 

The final differential is similar. The additive differential we can not import originates on a copy of $\F_{2^n}$. Explicitly, letting $\xi \in \F_{2^n}$ the additive differential is 
\[
d_{2^{n+1}-1}(\xi \bu^{2^n-1}u_{2\sigma}^{-2^{n-1}}a_{\sigma}^{2^{n+1}-1}) = \xi\bu^{2^{n+1}-2}u_{2\sigma}^{-2^n}a_{\sigma}^{2^{n+2}-2}
\]
The corresponding differential in the Picard spectral sequence is given by 
\[
\begin{split}
d_{2^{n+1}-1}^{\Pic}(\xi \bu^{2^n-1}u_{2\sigma}^{-2^{n-1}}a_{\sigma}^{2^{n+1}-1}) &= \xi \bu^{2^{n+1}-2}u_{2\sigma}^{-2^n}a_{\sigma}^{2^{n+2}-2} + \xi^2 \bu^{2^{n+1}-2}u_{2\sigma}^{-2^{n}}a_{\sigma}^{2^{n+2}-2}\\
&= (\xi + \xi^2)(\bu^{2^{n+1}-2}u_{2\sigma}^{-2^{n}}a_{\sigma}^{2^{n+2}-2}).
\end{split}
\]
This is zero whenever $\xi + \xi^2 = 0$, i.e, when $\xi = 0$ or $\xi = 1$. The kernel is thus $\Z/2$, as claimed. 
\end{proof}

We have now computed the stable range of the Picard spectral sequence. We are left with computing $H^0(C_2,\Pic(E_n))$ and $H^1(C_2,E_0^\times)$. The former is $\Z/2$, as computed by Baker and Richter \cite[Thm.~8.8]{baker_richter}. We thank Achim Krause for explaining the following. 
\begin{lem}\label{lem:h1}
	There is an isomorphism $H^1(C_2,E_0^{\times}) \cong \Z/2$.
\end{lem}
\begin{proof}
	The usual periodic resolution shows that for any abelian group $A$, the first group cohomology $H^1(C_2,A)$, where $C_2$ acts trivially on $A$, is given by the 2-torsion elements of $A$. A simple computation shows that in any integral domain $R$, the only non-trivial element with multiplicative order $2$ is $-1$. Indeed, suppose $x$ is such an element, then $(x-1)(x+1) = x^2 - 1 = 0$, so that $x = 1$ or $x = -1$. In particular, in the multiplicative group of units $R^{\times}$, the only non-trivial 2-torsion element is $-1$, so that $H^*(C_2,R^\times) \cong \Z/2$ generated by $-1$. The lemma then follows by taking $R = E_0$. 
\end{proof}
Putting this altogether, we obtain the calculation of $\Pic(E_n^{hC_2})$. 
\begin{thm}\label{thm:picardcalc}
	At the prime 2, and for any height $n$, there is an isomorphism
	\[
\Pic(E_n^{hC_2}) \cong \Z/2^{n+2},
	\]
	generated by $\Sigma E_n^{hC_2}$. 
\end{thm}
\begin{proof}
	Recall that $E_n^{hC_2}$ is $2^{n+2}$-periodic, so that this is a minimum bound on the order of the Picard group. On the other hand, we have computed that we have at most:
	\begin{itemize}
		\item A group of order $2$ in filtration degree 0. 
		\item A group of order $2$ in filtration degree 1. 
		\item A group of order $2$ in filtration degrees $2^k-1$ for $1 \le k \le n$. 
	\end{itemize}
	Together we see that we have a group of order at most $2^{n+2}$. Since this is also a lower bound, we conclude that $\Pic(E_n^{hC_2}) \cong \Z/(2^{n+2})$ generated by $\Sigma E_n^{hC_2}$. 
\end{proof}
We demonstrate the Picard spectral sequence for $n=2$ in \Cref{fig:picss}. 
\begin{rem}
  The previous theorem applies the following result. Let $\mathbb{S}_n  = \Aut(\Gamma_n)$ be the group of automorphisms of the Honda formal group law, so that $\mathbb{G}_n = \mathbb{S}_n \rtimes \Gal(\F_{p^n}/\F_p)$. 
   If $F \subset \G_n$ is a finite subgroup of the Morava stabilizer group such that $F \cap \mathbb{S}_n$ has 2-Sylow subgroup isomorphic to $C_2$, then $\Pic(E_n^{hF})$ is cyclic, generated by the suspension $\Sigma E_n^{hF}$. This follows from \Cref{thm:picardcalc} by applying \cite[Propositions 3.10 and 3.11]{hms_pic}. 
\end{rem}
\begin{figure}[hbt!]
\begin{center}
\subfloat{\scalebox{.8}{\includegraphics{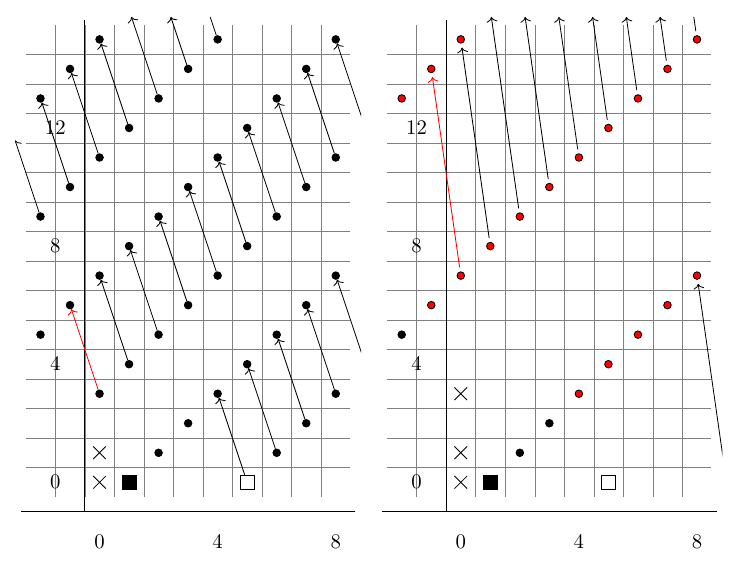}}}\\
\subfloat{\scalebox{.8}{\includegraphics{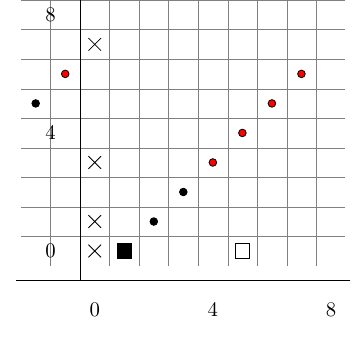}}}
  \caption{The $\E_3$, $\E_7$, and $\E_{\infty}$ pages for the Picard spectral sequence of $\pics(E_2^{hC_2})$ (some low dimensional non-contributing classes have been omitted). Here the squares denote copies of $\mathbb{W}(\F_4)[\![u_1]\!]$, black circles denote $\F_4[\![u_1]\!]$, red circles $\F_4$, crosses $\Z/2$, and black boxes $(E_0)^{\times}$. The differentials shown in red are those that cannot be imported from the additive spectral sequence.}\label{fig:picss}
  \end{center}
\end{figure}
\section{The Gross--Hopkins dual of \texorpdfstring{$E_n^{hC_2}$}{EnhC2} - first proof}
In this section we calculate the Gross--Hopkins dual $IE_n^{hC_2}$. We begin with an analysis of the Morava module of $E_n^{hF}$ for any finite subgroup $F \subseteq \G_n$. We then show that $E_n^{hC_2}$ is self-dual up to a shift. By \Cref{cor:recognition} to determine this shift we must find the gap in the homotopy groups of $IE_n^{hC_2}$. We do this by showing that a certain $C_2$-equivariant spectrum is strongly even. 
\subsection{The Morava module of \texorpdfstring{$IE_n^{hF}$}{IEnhF}}

In this section we calculate $\mE IE_n^{hF}$ as a Morava module, at any height $n$ and for any finite subgroup $F \subset \G_n$. We will not actually use this in our main computation, but include it as it may be of independent interest. In fact, under some conditions, one can use this to determine that $IE_n^{hF}$ is self-dual up to some suspension (as well as giving some control over what that suspension is), see \Cref{cor:dualmod4} below. 

To begin, we need the following results regarding homotopy fixed points and orbits in the $K(n)$-local category. Recall that the $K(n)$-local category has colimits, and they are given by taking the colimit in spectra, and then applying $K(n)$-localization again.
\begin{lem}\label{lem:groups}
Let $G$ be a finite group, and let $X$ be a spectrum with a $G$-action. 
  \begin{enumerate}
    \item There is an equivalence $L_{K(n)}(X_{hG}) \simeq L_{K(n)}((L_{K(n)}X)_{hG})$.
    \item If $X$ is $E_n$-local, then there is an equivalence $L_{K(n)}(X^{hG}) \simeq (L_{K(n)}X)^{hG}$.
    \item If $X$ is $E_n$-local, then the norm map
    \[
(L_{K(n)}X)_{hG} \to (L_{K(n)}X)^{hG}
    \]
    induces equivalences $L_{K(n)}(X_{hG}) \simeq (L_{K(n)}X)^{hG} \simeq L_{K(n)}(X^{hG})$. 
  \end{enumerate}
\end{lem}
\begin{proof}
  \begin{enumerate}
    \item Recall that $X_{hG}$ can be defined as the colimit $X_{hG} = \colim_G X$. The result then follows from the fact that $K(n)$-localization is left adjoint to the inclusion of $K(n)$-local spectra in all spectra, and the description of colimits in the $K(n)$-local category above. 
    \item By \cite[Cor.~6.1.3]{behrens_davis}, for any diagram $\{ X_i \}$ of $E_n$-local spectra, there is an equivalence $L_{K(n)}(\lim X_i) \simeq \lim (L_{K(n)}X_i)$. The result follows from the observation that $X^{hG}$ is defined as the limit $X^{hG} = \lim_G X$. 
    \item It is a consequence of \cite{gs_ate} that if $Y$ is a $K(n)$-local spectrum, then the norm map
    \[
Y_{hG} \to Y^{hG}
    \] 
is a $K(n)$-local equivalence. Taking $Y = L_{K(n)}X$ and applying (1) and (2) then gives the result. 
  \end{enumerate}
\end{proof}
\begin{lem}\label{lem:moravamoduleen}
  There is an equivalence of Morava modules 
  \[
  \mE(IE_n) \simeq \Sigma^{-n}\map_{cts}(\mathbb{G}_n,(E_n)_*\langle \det \rangle) \simeq \Sigma^{-n}\mE E\langle \det \rangle.
  \]

\end{lem}
\begin{proof}
  Since $I$ is dualizable in the $K(n)$-local category (see the proof of Prop.~17 of \cite{strickland_gross}), we have that $IE_n \simeq L_{K(n)}(DE_n \otimes I)$, where $DE_n = F(E_n,L_{K(n)}S^0)$ denotes the $K(n)$-local Spanier--Whitehead dual. It follows that $\mE(IE_n) \simeq \mE(DE_n \otimes I) \simeq \Sigma^{-n^2}\mE(E \otimes I)$, where we have used the equivalence $DE_n \simeq \Sigma^{-n^2}E_n$ \cite[Prop.~16]{strickland_gross}. Note that this equivalence is $\mathbb{G}_n$-equivariant in the homotopy category, and hence that this is an isomorphism of Morava modules. 

  By \Cref{thm:moravabc} we have $\mE I \cong \Sigma^{n^2-n}(E_n)_*\langle \det \rangle$. It follows from \cite[Prop.~8.4]{hs_99} that (as a $K_*$-module) $K_*I \cong \mE I/(p,u_1,\ldots,u_{n-1})  \cong \Sigma^{n^2-n}K_*$ is concentrated in even degrees. By \cite[Prop.~2.4]{ghmr} we deduce that there is an isomorphism of Morava modules $\mE(E \otimes I) \cong \map_{cts}(\G_n,\mE I)$. Together with \Cref{thm:moravabc} we see that 
  \[
\mE(IE_n) \cong \Sigma^{-n^2}\map_{cts}(\G_n,\mE I) \cong \Sigma^{-n}\map_{cts}(\G_n,(E_n)_*\langle \det \rangle).
  \]
  Finally, since $\mE E \cong \map_{cts}(\G_n,(E_n)_*)$, we see that $\mE E\langle \det \rangle \cong \map_{cts}(\G_n,(E_n)_*)\langle \det \rangle$, which is equivalent to $\map_{cts}(\G_n,(E_n)_*\langle \det \rangle)$, where $\map_{cts}(\G_n,(E_n)_*)$ is always given the diagonal $\G_n$-action. 
  \end{proof}
\begin{prop}\label{prop:moravamodulefinite}
  For any finite subgroup $F \subset \mathbb{G}_n$, there is an equivalence of Morava modules 
  \[
\mE(IE_n^{hF}) \cong \Sigma^{-n}\map_{cts}(\mathbb{G}_n/F,(E_n)_*\langle \det \rangle) \cong \Sigma^{-n} (\mE E_n^{hF})\langle \det \rangle
  \]
\end{prop}
\begin{proof}
  Consider the homotopy fixed point spectral sequence
  \[
H^*(F,\mE(IE_n)) \implies \pi_*(L_{K(n)}(E \otimes IE_n))^{hF},
  \]
  where the $F$-action is on $IE_n$. 

  Let us identify the abutment $(L_{K(n)}(E \otimes IE_n))^{hF}$ more carefully. Using \Cref{lem:groups} and the fact that homotopy orbits commute with smash product, there are equivalences
  \[
\begin{split}
  (L_{K(n)}(E \otimes IE_n))^{hF} & \simeq L_{K(n)}((E \otimes IE_n)_{hF}) \\
  & \simeq L_{K(n)}((E \otimes (IE_n)_{hF})) \\
  & \simeq L_{K(n)}((E \otimes L_{K(n)}(IE_n)_{hF})) \\
  & \simeq L_{K(n)}(E \otimes IE_n^{hF}).
\end{split}
  \]

  This identifies the abutment as $\mE(IE_n^{hF})$.

   By \Cref{lem:moravamoduleen} there is an isomorphism of Morava modules $\mE(IE_n)  \simeq \Sigma^{-n} \map_{cts}(\mathbb{G}_n,(E_n)_*\langle \det \rangle)$. This identifies the $\E_2$-term of the above spectral sequence as 
  \[
H^s(F,\map_{cts}(\mathbb{G}_n,\Sigma^{-n} E_{*}\langle \det \rangle)). 
  \]
  By the proof of \cite[Lem.~4.20]{dh_04}, this vanishes for $s > 0$, and so the spectral sequence collapses to show that 
  \[
\mE(IE_n^{hF}) \simeq \Sigma^{-n} \map_{cts}(\mathbb{G}_n/F,(E_n)_*\langle \det \rangle). 
  \]
  Now using \cite[Prop.~6.3]{dh_04} there is an equivalence \[\Sigma^{-n}(\mE E^{hF})\langle \det \rangle \cong \Sigma^{-n}\map_{cts}(\G_n/F,(E_n)_*)\langle \det \rangle \simeq \Sigma^{-n}\map_{cts}(\mathbb{G}_n/F,(E_n)_*\langle \det \rangle)\]
  where the last equivalence follows because we give $\map_{cts}(\G_n/F,(E_n)_*)$ the diagonal $\G_n$-action. 
\end{proof}
We now prove that the Gross--Hopkins duality of $E_n^{hF}$ is always an invertible $E_n^{hF}$-module. We thank Lennart Meier for the idea of how to show this. Before we begin the proof, we make the following standard observation. Let $(\cC,\otimes,\unit)$ be a closed symmetric monoidal category with internal hom $F_{\cC}(M,N)$ for $M,N \in \cC$. Then, there is a canonical evaluation map
\[
\epsilon \colon F_{\cC}(M,N) \otimes M \to N.
\]
given as the adjoint to the identity map on $F_{\cC}(M,N)$. 

Let $D_{\cC}(M) = F_{\cC}(M,\unit)$. Now suppose that $M$ is invertible, so that there exists $M^{-1}$ such that $M^{-1} \otimes M \simeq \unit$. Then, it is necessarily the case that $M^{-1} \simeq D_{\cC}(M)$, and that $\epsilon \colon D_{\cC}(M) \otimes M \to \unit$ is an isomorphism. The proof follows easily from the fact that invertible objects are dualizable, see e.g., \cite[Prop.~A.2.8]{HoveyPalmieriStrickl1997Axiomatic}. 

\begin{prop}\label{prop:grosshopkins_pic}
  For any finite subgroup $F \subset \G_n$, we have $IE_n^{hF} \in \Pic(E_n^{hF})$. 
\end{prop}
\begin{proof}
  We first observe that this is true when $F$ is the trivial group, i.e., that $IE_n \in \Pic(E_n)$. Indeed, since $\Pic(E_n)$ is algebraic, in the sense that $\Pic(E_n) \cong \Pic((E_n)_*)$, the Picard group of (graded) $(E_n)_*$-modules (see \cite[Theorem 9.1]{baker_richter}), it suffices to show that $\pi_*IE_n \cong (E_n)_*$ as $E_*$-modules, up to suspension. This is a consequence of \eqref{eq:detinen}, since $E_*\langle \det \rangle \cong E_*$, as an $E_*$-module. 

  For any finite subgroup $F \subset \G_n$, the morphism $E_n^{hF} \to E_n$ is a faithful Galois extension, see \cite[Prop.~3.6]{hms_pic}.  We recall that we consider $E_n$ as a cofree $F$-spectrum. In this case, there is a symmetric monoidal equivalence between the homotopy categories of $E_n^{hF}$-modules, and the homotopy category of genuine $F$-equivariant $E_n$-modules. Here the equivalence takes an equivariant $E_n$-module $M$, to the homotopy fixed point spectrum $M^{hF}$. Moreover, a map in the category of $F$-equivariant $E_n$-modules is an equivalence if it is an underlying equivalence of non-equivariant spectra.  In particular, there is an equivalence $\Pic_F(E_n) \cong \Pic(E_n^{hF})$ between the Picard group of $F$-equivariant $E_n$-modules, and the Picard group of $E_n^{hF}$-modules, given by taking homotopy fixed points. These claims can be found in the last paragraph of Section 6.1 of \cite{hm_tmf}, and are a consequence of Galois descent, first shown in this form by Meier \cite[Lemma 6.1.4 and Proposition 6.2.6]{meier_thesis}. We also refer the reader to \cite[Section 3]{bbhs_pic}, in particular Proposition 3.1 and Corollary 3.2. 

  Let $D_{E_n}(M)$ denote the dual of an $E_n$-module $M$ in the category of $E_n$-module spectra, so that $D_{E_n}(M) = F_{E_n}(M,E_n)$. As noted above $IE_n \in \Pic(E_n)$, so that in particular the evaluation map $\epsilon \colon D_{E_n}(IE_n) \otimes_{E_n} IE_n \to E_n$ is an equivalence. By naturality, this map is a map in the category of $F$-equivariant $E_n$-modules, and is an underlying equivalence, and hence by the discussion above we see that $IE_n \in \Pic_F(E_n)$. It follows from the equivalence $\Pic_F(E_n) \cong \Pic(E_n^{hF})$ that $(IE_n)^{hF} \cong I(E_n^{hF}) \in \Pic(E_n^{hF})$ (recall \Cref{rem:bracketing} - it does not matter if we take the homotopy fixed points inside or outside of the Gross--Hopkins dual). This completes the proof. 
\end{proof}
\subsection{The case \texorpdfstring{$F = C_2$}{F=C2}}
We now specialize to the case of interest, where the prime is 2, and $F = C_2$. In this section, we show that $IE_n^{hC_2}$ is self dual up to a shift, which is congruent to $n \pmod{4}$. We begin by showing that it is self-dual up to some shift. 
\begin{prop}\label{prop:grosshopkinspic}
  The Gross--Hopkins dual $IE_n^{hC_2} \simeq \Sigma^k E_n^{hC_2}$ for some integer $k$, which is only uniquely determined modulo $2^{n+2}$. 
\end{prop}

\begin{proof}
Combine \Cref{prop:grosshopkins_pic} with \Cref{thm:picardcalc}. 
\end{proof}

By \eqref{eq:detinen} we have $\pi_*IE \cong \Sigma^{-n}(E_n)_*\langle \det \rangle$ as Morava modules. We wish to know the structure of $(E_n)_*\langle \det \rangle$ as a $(E_n)_*[C_2]$-module.  
\begin{lem}\label{lem:c2det}
  As $(E_n)_*[C_2]$-modules, $(E_n)_*\langle \det \rangle \cong \Sigma^{1-(-1)^n}(E_n)_*$.
\end{lem}
\begin{proof}
  Let $\tau$ denote the generator of $C_2$, then $\det(\tau) = (-1)^n$ (it can be computed as the determinant of a $n \times n$ diagonal matrix with -1 along the diagonal).  It follows that, in $(E_n)_*\langle \det \rangle$, we have
  \[
\tau_*(u^k) = (-1)^{n+k}u^k
  \]
  In particular, if $n \equiv 0 \pmod{2}$, then $C_2$ is in the kernel of the determinant, so that the equivalence $(E_n)_* \langle \det \rangle \cong (E_n)_*$ is $C_2$-equivariant. 

  On the other hand, if $n \equiv 1 \pmod{2}$, then $u$ is invariant under the $C_2$-action. Multiplication by $u$ then gives the $C_2$-equivariant equivalence $(E_n)_*\langle \det \rangle \cong \Sigma^{2}(E_n)_*$. 
\end{proof}
By \Cref{prop:moravamodulefinite} we deduce the following. 
\begin{prop}\label{prop:moravamodulec2}
  There is an equivalence of $(E_n)_*[C_2]$-modules 
  \[
  \mE(IE_n^{hC_2}) \simeq \Sigma^{-n+1-(-1)^n}\mE(E_n^{hC_2}).
  \] 
\end{prop}

\begin{cor}\label{cor:dualmod4}
The Gross--Hopkins dual $IE_n^{hC_2} \simeq \Sigma^k E_n^{hC_2}$ for some integer $k \equiv n \pmod{4}$, which is only uniquely determined modulo $2^{n+2}$. 
\end{cor}
\begin{proof}
  Recall again that $\mE(E^{hC_2}) \simeq \map_{cts}(\G_n/C_2,(E_n)_*)$. An argument similar to that given on the bottom of page 286 of \cite{ghmr_pic} shows that the element $(u_{2\sigma }\bu^2)^j \in H^0(C_2,(E_n)_*)$ gives rise to an isomorphism $\mE(E^{hC_2}) \cong \mE(\Sigma^{4j} E^{hC_2})$ of $(E_n)_*[C_2]$-modules, for $j \in \Z$.

  We know that $IE_n^{hC_2} \simeq \Sigma^kE_n^{hC_2}$ for some $k \in \Z$. By \Cref{prop:moravamodulec2} and the previous paragraph we deduce that $k \equiv -n + 1 -(-1)^n \pmod{4}$. A simple calculation shows that $-n + 1 - (-1)^n \equiv n \pmod{4}$, and we are done. 
  \end{proof}
  \begin{rem}
    Let $p$ be odd, $n = p-1$, and let $F \subset \G_{p-1}$ denote the maximal finite subgroup of $\G_{p-1}$. In \cite[Prop.~4.13]{bbs_gross} it is shown that there is an isomorphism of $(E_n)_*[F]$-modules $(E_n)_*\langle \det \rangle \cong \Sigma^{2pj}(E_n)_*$, where $j = -n/2(n-2)$. Using the same type of arguments above, and the computation of the Picard group of $E_n^{hF}$-modules \cite{hms_pic}, one can deduce that  $IE_n^{hF} \simeq \Sigma^{k}E_n^{hF}$ where $k \equiv -n-pn(n-2) \pmod {2pn^2}$, a result also obtained in \cite{bbs_gross}. In \cite[Thm.~4.18]{bbs_gross} it is shown that $k = np^2 + n^2$, and indeed we have that $-n -pn(n-2) + 2pn^2 = np^2+n^2$. 
  \end{rem}
\subsection{The calculation of the Gross--Hopkins dual}
In this section we will compute the Gross--Hopkins dual $I(E_n^{hC_2})$. By the previous section we know that it is self-dual up to some suspension. Since $K(n)$-locally the norm map $(E_n)_{hC_2} \to E_n^{hC_2}$ is an equivalence, it follows that $(IE_n)^{hC_2} \simeq I(E_n^{hC_2})$. Hence, by \cref{prop:recognition} it suffices to find $j$ such that $\pi_j(IE_n)^{hC_2} = \pi_{j+1}(IE_n)^{hC_2} = \pi_{j+2}(IE_n)^{hC_2} = 0$. It follows immediately from the definitions that 
  \[
  \pi_t(IE_n)^{hC_2} \cong D_{\Q/\Z}(\pi_{-t}(M_nE_n)_{hC_2})
  \]
so it suffices to find when $\pi_{-j}(M_nE_{hC_2})=\pi_{-j-1}(M_nE_{hC_2})=\pi_{-j-2}(M_nE_{hC_2}) = 0$. 

We recall previously that we have constructed $E_n/I_k^{\infty}$ inductively by cofiber sequences
\[
E_n/I_{k-1}^{\infty} \to v_k^{-1}E_n/I_{k-1}^{\infty} \to E_n/I_k^{\infty}. 
\]

Our key computational input is the following. 
\begin{prop}\label{prop:vanishingstronglyeven}
  As a cofree $C_2$-spectrum, $E_n/I_n^{\infty}$ is strongly even, and hence $\pi_{i}(E_n/I_n^{\infty})_{hC_2} \cong \pi_{i}(E_n/I_n^{\infty})^{hC_2} = 0$ for $i = -3,-2,-1$. 
\end{prop}
The proof of this relies on the following lemma, which is due to Lennart Meier. 
\begin{lem}[Meier]\label{lem:stronglyeven}
  Suppose that $f\colon X \to Y$ is a morphism of strongly even $C_2$-spectra, with cofiber $C$. If $f^e_* \colon \pi_*^{e}X \to \pi_*^{e}Y$ is injective, then $C$ is strongly even as well. 
\end{lem}
\begin{proof}
  Recall that we must show that $\underline{\pi}_{k\rho-1}C = 0$ and that $\underline{\pi}_{k\rho}C$ is a constant Mackey functor for all $k \in \Z$. We observe that by \Cref{lem:greenstronglyeven} we have that $X$ and $Y$ are underlying even (i.e., $\pi^e_{2k-1}X = \pi^e_{2k-1}Y = 0$ for all $k \in \Z$) and $\pi^{C_2}_{\ast \rho -i}X$ and $\pi^{C_2}_{\ast \rho -i}Y$ are zero for $i = 1,2,3$.

We first show that $\underline \pi_{\ast\rho-1}C = 0$, i.e., that $\pi^{C_2}_{\ast\rho-1}C$ and $\pi^{e}_{\ast\rho-1}C $ are both trivial. The long exact sequence in homotopy and that the fact that $X$ and $Y$ are strongly even, so that in particular $\pi_{\ast\rho-2}^{C_2}X = 0$, shows that $\pi_{*\rho-1}^{C_2}C = 0$, so we just need to show that $\pi_{*\rho-1}^{e} C \cong \pi^e_{2*-1}C = 0$. This follows from the long exact sequence in homotopy, and the assumption that $f_*^e \colon \pi_{2*-2}^eX \to \pi^e_{2*-2}Y$ is injective. Finally, to see that $\underline \pi_{\ast\rho}C$ is a constant Mackey functor, we observe that it is the cokernel of the induced map $f_*\colon \underline \pi_{\ast\rho}X \to \underline \pi_{\ast\rho}Y$. 
\end{proof}
\begin{proof}[Proof of \Cref{prop:vanishingstronglyeven}]
  We prove inductively that $E_n/I_k^{\infty}$ is strongly even for $0 \le k \le n$. The base case is the claim that $E_n$ is strongly even, which is proved by Hahn and Shi, see \Cref{ex:stronglyeven}. Assume inductively that $E_n/I_{k-1}^{\infty}$ is strongly even. The property of being strongly even is closed under filtered colimits, and so $v_{k-1}^{-1}E_n/I_{k-1}^{\infty}$ is also strongly even. The cofiber sequence 
  \[
E_n/I_{k-1}^{\infty} \to v_{k-1}^{-1}E_n/I_{k-1}^{\infty} \to E_n/I_{k}^{\infty} 
  \]
  induces the short exact sequence 
  \[
0 \to (E_n)_*/I_{k-1}^{\infty} \to v_{k-1}^{-1}(E_n)_*/I_{k-1}^{\infty} \to (E_n)_*/I_k^{\infty} \to 0 
  \]
  on underlying homotopy groups. We see that the conditions of \Cref{lem:stronglyeven} are satisfied, and so the cofiber $E_n/I_{k}^{\infty}$ is strongly even as a $C_2$-spectrum. 

  The claim about the homotopy fixed points then follows from \Cref{lem:greenstronglyeven} and \Cref{cor:fixvshfp} in the case $k=n$. We also claim that $(E_n/I_n^{\infty})^{tC_2} \simeq \ast$, so that the norm $(E_n/I_n^{\infty})_{hC_2} \simeq (E_n/I_n^{\infty})^{hC_2}$ is an equivalence, completing the proof. This follows because it is a module over the spectrum $E_n^{tC_2}$, which is easily seen to be contractible, for example by the nilpotence of $a_{\sigma}$ in $\pi^{C_2}_*E_n$.
\end{proof} 

Using \Cref{prop:equivmonoequiv} we see that the homotopy orbit spectral sequence for $(M_nE_n)_{hC_2}$ is just a shift by $-n$ of the homotopy orbit spectral sequence for $(E_n/I_n^{\infty})_{hC_2}$, so we obtain the following corollary.  
\begin{cor}\label{cor:monovanishing}
  We have $\pi_i (M_nE_n)_{hC_2} = 0$ for $i = -3-n,-2-n,-1-n$. 
\end{cor}

Putting this altogether we obtain our computation of $IE_n^{hC_2}$. 
\begin{thm}\label{thm:grosshopkinshc2}
  The Gross--Hopkins dual $IE_n^{hC_2} \simeq \Sigma^{4+n}E_n^{hC_2}$. 
\end{thm}
\begin{proof}
  By \Cref{cor:monovanishing} we have $\pi_{i}(M_nE_n)_{hC_2}  = 0$ for $i = -3-n,-2-n,-1-n$. We observe that 
  \[
  \pi_iIE_n^{hC_2} \cong \pi_i (I_{\Q/\Z}M_nE_n)^{hC_2} \cong \pi_iI_{\Q/Z}((M_nE_n)_{hC_2}) \cong D_{\Q/\Z}\pi_{-i}(M_nE_n)_{hC_2},
\]
  so that $\pi_{1+n}IE_n^{hC_2} \cong \pi_{2+n}IE_n^{hC_2} \cong \pi_{3+n}IE_n^{hC_2} = 0$. By \Cref{prop:grosshopkinspic,cor:recognition} we deduce that $IE_n^{hC_2} \simeq \Sigma^{4+n}E_n^{hC_2}$. 
\end{proof}
\begin{rem}
  Note that $4+n \equiv n \pmod{4}$ as we expected from \Cref{cor:dualmod4}. 
\end{rem}
This has an equivariant refinement. 
\begin{thm}\label{thm:grosshopkinsequiv}
  The Gross--Hopkins dual $IE_n$ is $C_2$-equivariantly equivalent to $\Sigma^{4+n}E_n$. 
\end{thm}
\begin{proof}
  This is essentially the same argument as given in \cite[Thm.~8.2]{anderson_ko}. By \eqref{eq:detinen} and periodicity there is a non-equivariant equivalence $\pi_*(\Sigma^{4+n}E_n) \cong \pi_*(IE_n)$. In the usual way, see for example \cite[Prop.~2.2]{anderson_ko}, there is non-equivariant equivalence of spectra $\Sigma^{4+n}E_n \xr{\simeq} IE_n$, which is nonetheless an $E_n$-module map. Let $d \colon S^{4+n} \to IE_n$ be the spectrum map adjoint to this. 

 Now consider the $C_2$-cofiber sequence 
  \[
C_{2+}\to S^0 \to S^{\sigma},
  \]
  which gives rise to an exact sequence 
  \[
[S^{4+n},IE_n]_{C_2} \to [S^{4+n},IE_n] \to [S^{3+n+\sigma},IE_n]_{C_2}. 
  \]
  It suffices to show that $d \in [S^{4+n},IE_n]$ maps to zero in $[S^{3+n+\sigma},IE_n]_{C_2}$. But by periodicity and \Cref{cor:fixvshfp} 
  \[
\pi_{3+n+\sigma}^{C_2}IE_n \cong \pi^{C_2}_{2+n}IE_n \cong \pi_{2+n}IE_n^{hC_2}.
  \]
By \Cref{thm:grosshopkinshc2} this group is isomorphic to $\pi_{-2}E_n^{hC_2}$ which is trivial by \Cref{prop:recognition}. 
\end{proof}
\section{The Gross--Hopkins dual of \texorpdfstring{$E_n^{hC_2}$}{EnhC2} - second proof}
In this section, we give a second proof for the computation of $IE_n^{hC_2}$. We do this by completely computing the $C_2$-homotopy fixed point spectral sequence of $IE_n$. More specifically, we will compute $IE_n^{hC_2}$ by first computing the $C_2$-Tate spectral sequence of $E_n/I_k^\infty$ inductively on $k$.  This approach has the advantage that it does not require knowledge of the Picard group of $E_n^{hC_2}$-modules. The base case, when $k = 0$, is a direct consequence of the calculations of Hahn and Shi given in \Cref{thm:additivess}.  For the inductive step, we will transport differentials along maps of spectral sequences that are induced from the cofiber sequence
$$E_n/I_k^\infty \longrightarrow v_k^{-1}E_n/I_k^\infty \longrightarrow E_n/I_{k+1}^\infty.$$
To transport these differentials, we prove a version of the generalized geometric boundary theorem of Mark Behrens in the appendix (see \Cref{geometricboundary}). The theorem is tailored for our purposes and we prove it by universal model methods, which is different from Behrens' original diagram chasing approach.

Once we have computed the $C_2$-Tate spectral sequence of $E_n/I_n^\infty$, we can immediately determine the $C_2$-homotopy fixed point spectral sequence of $I_nE_n$ by using \Cref{lem:qzdual,prop:equivmonoequiv}.  The computation of the HFPSS$(IE_n^{hC_2})$, together with the $E_n^{hC_2}$-module structure on $IE_n^{hC_2}$, will imply that $IE_n^{hC_2} \simeq \Sigma^{n+4}E_n^{hC_2}$.

\subsection{Computations}
We will work with the Tate spectral sequence, which has the form
\[
\hat{H}^s(C_2, \pi_t E_n) \Longrightarrow \pi_{t-s}E_n^{tC_2},
\] 
where $\hat{H}^*(C_2, \pi_* E_n)$ denotes the Tate cohomology of $C_2$ with coefficients in $\pi_* E_n$. We will import the $RO(C_2)$-graded result from \Cref{thm:additivess} and then focus on the integer graded part. There are maps of spectral sequences
\[
   \setlength\mathsurround{0pt}
\begin{tikzcd}
H^s(C_2, \pi_t E_n) \arrow[r] \arrow[d,Rightarrow,"\text{HFPSS}"] & \hat{H}^s(C_2, \pi_t E_n) \arrow[r] \arrow[d,Rightarrow,"\text{Tate SS}"] & H_{-s-1}(C_2,\pi_t E_n) \arrow[d,Rightarrow, "\text{HOSS}"]\\
\pi_{t-s}E_n^{hC_2} \arrow[r] & \pi_{t-s}E_n^{tC_2} \arrow[r] & \pi_{t-s-1}(E_n)_{hC_2}
\end{tikzcd}
\]
such that the first map is multiplicative \cite{green_may_tate}. The maps in the top row of the diagram (i.e., the maps between $\E_2$-pages) are isomorphisms when $s>0$ and $s<-1$, respectively.  
For $s=0,-1$, we have the exact sequence
$$0\rightarrow \hat{H}^{-1}(C_2, \pi_t E_n) \rightarrow H_0(C_2, \pi_t E_n) \xrightarrow{N} H^0(C_2, \pi_t E_n) \rightarrow \hat{H}^0(C_2, \pi_t E_n) \rightarrow 0$$
where $N$ is the algebraic norm map. In particular, the differentials in the HFPSS and the HOSS will be completely determined by the differentials in the Tate SS in this case. 

\begin{prop}
  \begin{enumerate}
    \item The $\E_2$-page of the $RO(C_2)$-graded Tate spectral sequence for $E_n$ is
  \[
\E_2^{s, t} = W(\mathbb{F}_{2^n})[\![\bar{u}_1, \bar{u}_2, \ldots, \bar{u}_{n-1}]\!][\bar{u}^{\pm 1}] \otimes \mathbb{Z}[u_{2\sigma}^{\pm 1}, a_\sigma^{\pm 1}]/(2a_\sigma).
\]
Here, the bidegrees of the classes are as follows: $|\bu_i| = (0,0)$ for all $1 \leq i \leq n-1$, $|\bar{u}| = (\rho,0)$, $|u_{2\sigma}| = (2-2\sigma,0)$ and $|a_{\sigma}| = (-\sigma,1)$. 
\item The classes $\bar{u}_1$, $\ldots$, $\bar{u}_{n-1}$, $\bar{u}^{\pm 1}$, and $a_\sigma$ are permanent cycles.  All the differentials in the spectral sequence are determined by the differentials 
\begin{eqnarray*}
d_{2^{k+1} -1} (u_{2\sigma}^{2^{k-1}}) &=&  \bar{u}_k\bar{u}^{2^k-1}a_\sigma^{2^{k+1}-1}, \, \, \, 1 \leq k \leq n-1, \\ 
d_{2^{n+1}-1}(u_{2\sigma}^{2^{n-1}})&=& \bar{u}^{2^n-1}a_\sigma^{2^{n+1}-1}, \, \, \, k = n, \
\end{eqnarray*}
and multiplicative structures. 
\end{enumerate}
\end{prop}
\begin{proof}
The $\E_2$-page is obtained by a direct computation. Under the spectral sequence map from the HFPSS($E_n^{hC_2}$) to the Tate SS($E_n^{tC_2}$), the elements 
$$\bu_i, \bar{u}, u_{2\sigma}, a_{\sigma}$$
on the $\E_2$-page of the HFPSS map to the elements 
$$\bu_i, \bar{u}, u_{2\sigma}, a_{\sigma}$$
on the $\E_2$-page of the Tate SS, respectively. By \Cref{thm:additivess}, in the HFPSS we have the differentials
\begin{eqnarray*}
d_{2^{k+1} -1} (u_{2\sigma}^{2^{k-1}}) &=&  \bar{u}_k\bar{u}^{2^k-1}a_\sigma^{2^{k+1}-1}, \, \, \, 1 \leq k \leq n-1, \\ 
d_{2^{n+1}-1}(u_{2\sigma}^{2^{n-1}})&=& \bar{u}^{2^n-1}a_\sigma^{2^{n+1}-1}, \, \, \, k = n. \
\end{eqnarray*}
By the naturality of the spectral sequence map, we deduce the desired differentials in the Tate SS. Furthermore, since the classes $\bar{u}_1$, $\ldots$, $\bar{u}_{n-1}$, $\bar{u}^{\pm 1}$, and $a_\sigma$ are permanent cycles in the HFPSS, their images in the Tate SS are also permanent cycles.

The Tate SS is multiplicative and all the elements are of the form $u_{2\sigma}^k$ times a permanent cycle. This, combined with the differentials above, give all the differentials in the Tate SS.  After these differentials, the spectral sequence will vanish on the $\E_{2^{n+1}}$-page.
\end{proof}

Taking the integer-graded part of the above spectral sequence, we obtain the following corollary.
\begin{cor}{\label{integer}}
Let $u^2:=u_{2\sigma}\bar{u}^2$ and $\alpha := a_{\sigma} \bar{u}$.  The $\E_2$-page of the integer-graded $C_2$-Tate spectral sequence of $E_n$ is
 $$\E_2^{s,t}=(\pi_0E_n/2)[u^{\pm 2}, \alpha^{\pm 1}],$$
where the bidegrees $(t-s,s)$ of the elements are the following: 
\begin{eqnarray*}
|u^2|&=&(4,0), \\
|\alpha|&=&(1,1).
\end{eqnarray*}
All the differentials are determined via multiplicative structures by the differentials 
\begin{eqnarray*}
d_{2^{l+1}-1}(u^{2^l}) &=& u_l \alpha^{2^{l+1}-1},\,\,\, 1 \leq l \leq n-1, \\ 
d_{2^{n+1}-1}(u^{2^n}) &=& \alpha^{2^{n+1}-1}.
\end{eqnarray*}
  \end{cor}
\begin{rem}
The element $u^2$ coincides with the square of $u$ where $u \in (E_n)_2$.
\end{rem}
\begin{rem}
  The previous two computations show that $E_n^{tC_2} \simeq \ast$. As noted previously, this is also a consequence of the fact that $E_n^{hC_2} \to E_n$ is a faithful Galois extension \cite[Prop.~3.6]{hms_pic}. 
\end{rem}
This will be our starting point to compute the $C_2$-Tate spectral sequence for $E_n/2^\infty = E_n/I_1^\infty$, which is the $k=1$ case of the main theorem in this section (\Cref{inductive}, computing the Tate SS for $E_n/I_k^\infty$).  We make the convention $u_n=1$ so that the differentials can be described as
\begin{eqnarray*}
d_{2^{l+1}-1}(u^{2^l}) &=& u_l \alpha^{2^{l+1}-1},\,\,\, 1 \leq l \leq n.
\end{eqnarray*}
To describe the $\E_2$-page of the Tate SS for $E_n/I_k^\infty$, we pause here to introduce some algebraic notations.
\begin{defn}
Let $R$ be a commutative ring, and $r$ a nonzero divisor in $R$. We denote the cokernel of $R \hookrightarrow r^{-1}R$ by $R/r^\infty$.
\end{defn}
The cokernel $R/r^\infty$ is a non unital ring and an $R$-module. An element of $R/r^\infty$ can be represented as $\frac{a}{r^k}$, where $a\in R$ is not divisible by $r$ and $k>0$. The multiplication is given by $\frac{a}{r^{k_1}} \cdot \frac{b}{r^{k_2}}=\frac{ab}{r^{k_1+k_2}}$. For the module structure, given $a \in R$, $\frac{b}{r^k} \in R/r^\infty$, then $a \cdot \frac{b}{r^k}=\frac{ab}{r^k}$. Suppose that $a=a'r^{k'}$ where $a'$ is not divisible by $r$, then $\frac{ab}{r^k}=\frac{a'b}{r^{k-k'}}$. Note that if $k' \geqslant k$, then $\frac{a'b}{r^{k-k'}}=0$ in $R/r^\infty$.
\begin{ex}
At prime $2$, $(E_n)_*/2^\infty$ is the cokernel of $(E_n)_* \hookrightarrow 2^{-1}(E_n)_*$ and $(E_n)_*/(2^\infty, u_1^\infty)$ is the cokernel of $(E_n)_*/2^\infty \hookrightarrow u_1^{-1}(E_n)_*/2^\infty$. We denote $(E_n)_*/(2^\infty, u_1^\infty)$ by $(E_n)_*/I_2^\infty$. In $(E_n)_*/I_2^\infty$, the element $2 \cdot \frac{u^2}{2u_1}$ is zero because it is equal to $\frac{u^2}{2^0u_1}$, and the power of $2$ in the denominator is $0$, which is not positive. Recall that in \Cref{subsection:mono}, we denote $E_n\otimes S^0/I_k^\infty$ by $E_n/I_k^\infty$. The homotopy groups $\pi_*(E_n/I_k^\infty)$ (also denoted by $(E_n/I_k^\infty)_*$) are
$$(E_n)_*/I_k^\infty=(E_n)_0/(2^\infty, u_1^\infty,\cdots,u_{k-1}^\infty)[u^{\pm 1}].$$
In particular, they are $(E_n)_*$-modules. 
\end{ex}

Since the spectrum $E_n/I_k^\infty$ is an $E_n$-module, the $C_2$-Tate spectral sequence for $E_n/I_k^\infty$ is a module over the the $C_2$-Tate spectral sequence of $E_n$. In particular, the $\E_2$-page of the $C_2$-Tate spectral sequence of $E_n/I_k^\infty$ is a module over the $\E_2$-page of the $C_2$-Tate spectral sequence of $E_n$ (which we denote by $\E_2(E_n^{tC_2})$). For a $\E_2(E_n^{tC_2})$-module $M$, we use the notation $M\cdot u$ for the module that is isomorphic to $M$, and every element is named as $mu$ where $m \in M$ and $u \in (E_n)_2$.

\begin{lem}\label{lem:E2pageTateEn/Ik}
For $1 \le k \le n$, the $\E_2$-page of the $C_2$-Tate spectral sequence for $E_n/I^\infty_k$ is
    $$\E_2^{*,*} =\hat{H}^*(C_2, (E_n/I^\infty_k)_*) = ((E_n/2)_0/(u_1^\infty,\cdots,u_{k-1}^\infty)[\alpha^{\pm 1},u^{\pm 2}]) \cdot u$$
where the $(t-s,s)$-bidegree of $u$ is $(2,0)$ and the $(t-s,s)$-bidegree of $\alpha$ is $(1,1)$.
\end{lem}
\begin{proof}
We compute the $\E_2$-page inductively on $k$. The base case is the $\E_2$-page of the $C_2$-Tate spectral sequence for $E_n/I^\infty_0$, which is just $E_n$, and is stated in \Cref{integer}. Now, assume that the statement is true for $k-1<n$. The short exact sequence
$$E_n/I_{k-1}^\infty \longrightarrow u_{k-1}^{-1}E_n/I_{k-1}^\infty \longrightarrow E_n/I_k^\infty$$
induces the long exact sequence
$$\cdots \hat{H}^s(C_2,  (E_n/I^\infty_{k-1})_*) \xrightarrow{i} \hat{H}^s(C_2, (u_{k-1}^{-1}E_n/I^\infty_{k-1})_*) \xrightarrow{p} \hat{H}^s(C_2, (E_n/I^\infty_k)_*) \xrightarrow{\partial} \hat{H}^{s+1}(C_2, (E_n/I^\infty_{k-1})_*) \rightarrow \cdots.$$
Because, as an $(E_n)_*$-module, $(E_n/I^\infty_{k-1})_*$ is $u_{k-1}$-torsion free, where $u_{k-1} \in (E_n)_0$, the second term in the long exact sequence above is 
$$\hat{H}^s(C_2, (u_{k-1}^{-1}E_n/I^\infty_{k-1})_*)=u_{k-1}^{-1}\hat{H}^s(C_2, (E_n/I^\infty_{k-1})_*)$$
and the map $i$ is the inclusion into the $u_{k-1}$-localization. This implies that in the long exact sequence above, the map $\partial$ is the zero map in all degrees, and the long exact sequence splits into short exact sequences. 
In the short exact sequence
\begin{equation}\label{equation:ses}
0\rightarrow \E_2^{s,t}((E_n/I^\infty_{k-1})^{tC_2}) \rightarrow \E_2^{s,t}((u_{k-1}^{-1}E_n/I^\infty_{k-1})^{tC_2}) \rightarrow \E_2^{s,t}((E_n/I^\infty_{k})^{tC_2}) \rightarrow 0,
\end{equation}
the first map is the inclusion map
$$i \colon \E_2^{s,t}((E_n/I^\infty_{k-1})^{tC_2})\lhook\joinrel\longrightarrow u_{k-1}^{-1}\E_2^{s,t}((E_n/I^\infty_{k-1})^{tC_2}),$$
so the third term is
\begin{equation}\label{equation:third}
\E_2^{s,t}((E_n/I^\infty_{k})^{tC_2})=\coker(i)=\E_2^{s,t}((E_n/I^\infty_{k-1})^{tC_2})/u_{k-1}^\infty.\end{equation}
By the inductive assumption, we have 
$$\E_2^{s,t}((E_n/I^\infty_{k-1})^{tC_2})= ((E_n/2)_0/(u_1^\infty,\cdots,u_{k-2}^\infty)[\alpha^{\pm 1},u^{\pm 2}]) \cdot u.$$ 
Therefore, from \eqref{equation:third}, the $\E_2$-page of the $C_2$-Tate spectral sequence of $E_n/I_k^\infty$ is
$$\E_2^{s,t} = ((E_n/2)_0/(u_1^\infty,\cdots,u_{k-1}^\infty)[\alpha^{\pm 1},u^{\pm 2}]) \cdot u.$$
This proves the claim for $k$ and completes the inductive step.

\end{proof}
\begin{rem}
If we only want the additive structure (and not the $\E_2(E_n^{tC_2})$-module structure), we can compute $\hat{H}^s(C_2, (E_n/I^\infty_k)_t)$ directly. Since $C_2$ acts trivially on $(E_n/I^\infty_k)_t$ when $t=4m$ and acts as multiplication by $(-1)$ on $(E_n/I^\infty_k)_t$ when $t=4m+2$, the $C_2$-Tate cohomology of $E_n/I^\infty_k$ is
\[
\hat{H}^s(C_2, (E_n/I^\infty_k)_t)=
\begin{cases}
\kernel(E_n/I^\infty_k \xrightarrow{\times 2} E_n/I^\infty_k)_t &\text{for odd }s\text{ and }4\,|\,t;\\
\kernel(E_n/I^\infty_k \xrightarrow{\times 2} E_n/I^\infty_k)_t &\text{for even }s\text{ and }2\,|\,t \text{ but }4 \nmid t;\\
0 &\text{else}.
\end{cases}
\]
The kernel $\kernel(E_n/I^\infty_k \xrightarrow{\times 2} E_n/I^\infty_k)$,
as a set, is
$$\left\{\frac{x}{2u_1^{\ell_1}\cdots u_{k-1}^{\ell_{k-1}}} \mid x \in E_n, \text{ and not divisible by }2,u_1,\cdots,u_{k-1}\right\}.$$
As a non unital ring and $E_n$-module, it is
$$\kernel(E_n/I^\infty_k \xrightarrow{\times 2} E_n/I^\infty_k) \cong (E_n/2)_0/(u_1^\infty,\cdots,u_{k-1}^\infty).$$
\end{rem}

Denote $\frac{1}{2u_1\cdots u_{k-1}}$ by $x_k$. With this notation, we can write any element in $E_n/I_k^\infty$ as $y x_k^m$ where $y \in E_n$ and $m>0 \in \mathbb{Z}$.  For example, the element
$$\frac{u^3}{2u_1^2} \in \kernel(E_n/I^\infty_2 \xrightarrow{\times 2} E_n/I^\infty_2)$$
is $2u^3x_2^2$.  In particular, any element in $\kernel(E_n/I^\infty_k \xrightarrow{\times 2} E_n/I^\infty_k)$ will have the format $2^{m-1}y x_k^m$ where $y \in E_n$ and $m$ a positive integer.  Let
$$x_k(m)=2^{m-1} x_k^m.$$
With this notation, elements on the $\E_2$-page of the $C_2$-Tate spectral sequence of $E_n/I^\infty_k$ can be written as $yx_k(m)$.

\begin{thm}{\label{inductive}}
Let $j$ be an integer and $m$ be a positive integer. All the nonzero differentials in the $C_2$-Tate spectral sequence for $E_n/I^\infty_k$ are $(E_n/2)_0[\alpha]$-linear and are determined by the following differentials:
\begin{enumerate}
\item[(1)] for $1 \leqslant l < k$,
$$d_{2^{l+1}-1} \left(u^{2^{l+1}j+2^{l+1}-1}x_k(m)\right) = u_l  \alpha^{2^{l+1}-1} u^{2^{l+1}j+2^l-1}x_k(m);$$
\item[(2)] for $k \leqslant l \leqslant n$,
$$d_{2^{l+1}-1} \left(u^{2^{l+1}j+2^l+2^k-1}x_k(m)\right) =  u_l \alpha^{2^{l+1}-1} u^{2^{l+1}j+2^k-1}x_k(m). $$

  \end{enumerate}
\end{thm}
 
Before giving the proof, note that setting $k=n$ in \Cref{lem:E2pageTateEn/Ik} and \Cref{inductive} produces the following corollary:
    \begin{cor}{\label{monocomputation}}
The $C_2$-Tate spectral sequence of $E_n/I^\infty_n$ has $\E_2$-page
    $$\E_2^{s,t} = (E_n/2)_0/(u_1^\infty,\cdots,u_{n-1}^\infty)[\alpha^{\pm 1},u^{\pm 2}]\cdot u.$$

All the nonzero differentials are $(E_n/2)_0[\alpha]$-linear and are determined by the differentials 
    \[
    \begin{split}d_{2^{l+1}-1} \left(u^{2^{l+1}j+2^{l+1}-1}x_n(m)\right) &=  u_l u^{2^{l+1}j+2^l-1} \alpha^{2^{l+1}-1}x_n(m),\,\,\,1 \leqslant l \leqslant n,
     \end{split}
     \]
where $j,m \in \mathbb{Z}$, $m>0$.
    \end{cor}

\begin{rem}
Despite the complexity of the formulas, the differentials are obtained by repeatedly using \Cref{geometricboundary} to transport differentials from the $C_2$-Tate spectral sequence of $E_n/I_{k-1}^\infty$ to the Tate spectral sequence of $E_n/I_k^\infty$. They form a regular pattern. This pattern can be seen in \Cref{example}, where the entire computation is given for the case $n=2$.  
\end{rem}

\begin{proof}[Proof of \Cref{inductive}]
       
We have computed the $\E_2$-page in \Cref{lem:E2pageTateEn/Ik}.  To prove the differentials, we will use induction on $k$. Note that the $C_2$-Tate spectral sequence of $E_n/I_k^\infty$ is a module over the $C_2$-Tate spectral sequence of $E_n$, and $(E_n)_0[\alpha]$ are permanent cycles in the $C_2$-Tate spectral sequence of $E_n$. Therefore, all differentials in the $C_2$-Tate spectral sequence of $E_n/I_k^\infty$ are $(E_n)_0[\alpha]$-linear.

The base case, when $k=1$, can be deduced from \Cref{integer} as follows. Consider the cofiber sequence
$$E_n \rightarrow 2^{-1}E_n \rightarrow E_n/I_1^\infty.$$
By \Cref{integer}, the differentials in the $C_2$-Tate spectral sequence of $E_n$ are determined by the differentials 
$$d_{2^{l+1}-1}(u^{2^l}) = u_l \alpha^{2^{l+1}-1},\, \, \, 1 \leqslant l \leqslant n$$
and multiplicative structures. In particular, $\alpha$ is a permanent cycle and the differentials are $\alpha$-linear. The $\E_2$-pages form a long exact sequence with connecting morphism $\partial$ as follows:
  $$\cdots \rightarrow \E_2^{s,t}(E_n^{tC_2}) \rightarrow \E_2^{s,t}((2^{-1}E_n)^{tC_2}) \rightarrow \E_2^{s,t}((E_n/I_1^\infty)^{tC_2}) \overset{\partial}{\rightarrow} \E_2^{s+1,t}(E_n^{tC_2}) \rightarrow \cdots.$$
By the classical geometric boundary theorem, e.g., \Cref{geometricboundary} with $r=1$, we deduce the differentials 
\begin{align*}
\partial \left(d_{2^{l+1}-1}\left(u^{2^l+1}x_1\right)\right) &= d_{2^{l+1}-1}\left(\partial\left(u^{2^l+1}x_1\right)\right) \\&= d_{2^{l+1}-1}(\alpha u^{2^l}) \\&= u_l \alpha^{2^{l+1}} \\ &= \partial \left(u_l \alpha^{2^{l+1}-1}u x_1\right),\, \, \, 1 \leqslant l \leqslant n
\end{align*}
Note that in the formula above, $x_1$ is $\frac{1}{2}$ by definition. 

Because everything on the $\E_2$-page of the $C_2$-Tate spectral sequence of $E_n$ is $2$-torsion, the $\E_2$-page of the $C_2$-Tate spectral sequence of $2^{-1}E_n$ vanishes at the $\E_2$-page, and the connecting morphism $\partial$ is an isomorphism.  Therefore, from the differential above, we obtain the differentials 
$$d_{2^{l+1}-1}\left(u^{2^l+1}x_1\right) = u_l \alpha^{2^{l+1}-1}ux_1,\, \, \, 1 \leqslant l \leqslant n$$
in the $C_2$-Tate spectral sequence of $E_n/I_1^\infty$.
With the same argument, we have the differentials 
$$d_{2^{l+1}-1} \left(u^{2^{l+1}j+2^l+1}x_1(m)\right) = u_l u^{2^{l+1}j+1} \alpha^{2^{l+1}-1}x_1(m), \,\,\, 1 \leqslant l \leqslant n$$
in the $C_2$-Tate spectral sequence of $E_n/I_1^\infty$. For degree reasons, these differentials determine all the nonzero differentials by $(E_n/2)_0[\alpha]$-linearity.
This finishes the computation for the $k=1$ case. 
  
Now, suppose that \Cref{inductive} holds for $k$ (where $k<n$).  To prove the claim for $k+1$, consider the cofiber sequence
$$E_n/I_k^\infty \rightarrow u_k^{-1}E_n/I_k^\infty \rightarrow E_n/I_{k+1}^\infty.$$
The differentials in the $C_2$-Tate spectral sequence of $E_n/I_k^\infty$ are given by the inductive hypothesis.  By naturality, we can deduce the following differentials in the Tate spectral sequence of $u_k^{-1}E_n/I_k^\infty:$
$$d_{2^{l+1}-1} \left(u^{2^{l+1}j+2^{l+1}-1}x_k(m)\right) = u_l u^{2^{l+1}j+2^l-1} \alpha^{2^{l+1}-1}x_k(m),  \,\,\, 1 \leqslant l \leqslant k.$$

For degree reasons, the above differentials determines all the nonzero differentials in the Tate spectral sequence for $u_k^{-1}E_n/I_k^\infty$ by $(E_n/2)_0[\alpha]$-linearity.  The $\E_r$-page ($r< 2^{k+1}$) of the spectral sequences consists of lines of slope 1 passing $(4m+2,0)$ and each line either supports differentials to another line or is hit by differentials from another line.  Hence, there is no room for other $d_r$-differentials for $r<2^{k+1}$.  At the $\E_{2^{k+1}-1}$-page, all the dots that are left are free rank one $(u_k^{-1}E_n/(2, u_1,\cdots, u_{k-1}))_0$-modules. The $d_{2^{k+1}-1}$-differentials are multiplication by $u_k$. Since $u_k$ is invertible in $(u_k^{-1}E_n/(2, u_1,\cdots, u_{k-1}))_0$, all the $d_{2^{k+1}-1}$-differentials are isomorphisms from the sources to the targets.  On the $\E_{2^{k+1}-1}$-page, every line is either supporting or hit by a differential which is an isomorphism.  Therefore, after these differentials, the $\E_{2^{k+1}}$-page vanishes and the spectral sequence collapses.
  
By naturality, the following differentials exist in the Tate spectral sequence of $E_n/I_{k+1}^\infty$:
$$d_{2^{l+1}-1} \left(u^{2^{l+1}j+2^{l+1}-1}x_k(m)\right) = u_l u^{2^{l+1}j+2^l-1} \alpha^{2^{l+1}-1}x_k(m),\,\,\, 1 \leqslant l \leqslant k.$$
These differentials are $(E_n/2)_0[\alpha]$-linear in the $C_2$-Tate spectral sequence of $E_n/I_{k+1}^\infty$. In particular, the above differentials determine all the nonzero differentials up to the $\E_{2^{k+1}}$-page. This proves the differentials in (1). 

To prove the differentials in (2), note that by the inductive hypothesis, we have the following differentials in the Tate spectral sequence of $E_n/I_k^\infty$: 
$$d_{2^{l+1}-1} \left(u^{2^{l+1}j+2^l+2^k-1}x_k(m)\right) = u_l u^{2^{l+1}j+2^k-1} \alpha^{2^{l+1}-1}x_k(m), \,\,\, k +1 \leqslant l \leqslant n. $$

We will now apply \Cref{geometricboundary} to the sequence
$$E_n/I_k^\infty \overset{i}{\longrightarrow} u_k^{-1}E_n/I_k^\infty \overset{p}{\longrightarrow} E_n/I_{k+1}^\infty,$$
with $r=2^{k+1}-1$, $r'=2^{l+1}-1$, and 
\begin{eqnarray*}
x&=& u_k u^{2^{l+1}j+2^l+2^k-1}x_k(m), \\ 
x'&=& u_lu_k u^{2^{l+1}j+2^k-1} \alpha^{2^{l+1}-1}x_k(m), \\
y_1'&=&u_k u^{2^{l+1}j+2^l+2^k-1}x_k(m), \\
y_1&=&u^{2^{l+1}j+2^l-1} \alpha^{-2^{k+1}+1}x_{k}(m),\\
y_2'&=&u_lu_k u^{2^{l+1}j+2^k-1} \alpha^{2^{l+1}-1}x_k(m), \\
y_2&=&u_l u^{2^{l+1}j+2^{k+1}-1} \alpha^{2^{l+1}-2^{k+1}}x_k(m), \\
z&=&u^{2^{l+1}j+2^l+2^{k+1}-1} \alpha^{-2^{k+1}+1}x_{k+1}(m), \\
z'&=& u_l u^{2^{l+1}j+2^{k+1}-1} \alpha^{2^{l+1}-2^{k+1}}x_{k+1}(m).
\end{eqnarray*}
It is straightforward to check that 
\begin{eqnarray*}
d_{r'}x&=&x', \\
d_ry_1&=&y_1', \\
d_ry_2&=&y_2', \\
i_*(x)&=&y_1', \\
i_*(x')&=&y_2', \\
p_*(y_1)&=&z, \\
p_*(y_2)&=&z'. 
\end{eqnarray*}
It follows from \Cref{geometricboundary} that $d_{r'}z=z'$.  Therefore, we have the differentials
\[
d_{2^{l+1}-1}\left(u^{2^{l+1}j+2^l+2^{k+1}-1} \alpha^{-2^{k+1}+1}x_{k+1}(m)\right) 
=u_l u^{2^{l+1}j+2^{k+1}-1} \alpha^{2^{l+1}-2^{k+1}}x_{k+1}(m)
\]
for $k+1 \leq l \leq n$. One can multiply everything by powers of $\alpha$ and run the same argument to prove that all these differentials are $\alpha$-linear. The differentials in (2) follow directly from these differentials by multiplying both sides by $\alpha^{2^{k+1}-1}$.  This completes the inductive step.          
    \end{proof}
    \begin{ex}\label{example}
    In this example, we will demonstate the inductive computation from the $C_2$-Tate spectral sequence of $E_2$ to the $C_2$-Tate spectral sequence of $E_2/2^\infty = E_2/I_1^\infty$, and then from the $C_2$-Tate spectral sequence of $E_2/I_1^\infty$ to the $C_2$-Tate spectral sequence of $E_2/I_2^\infty$. The input is the $C_2$-Tate spectral sequence of $E_2$ (see the $n=2$ case in \Cref{integer}). We use charts to show the computational results. The meaning of the dots should be clear from the formula in \Cref{inductive} (e.g., the black dots in \Cref{fig:TateSSE2/pinfty-page3} denote copies of $(\pi_0(E_2)/2^\infty)[[u_1]]$ while black dots in Figure \Cref{fig:TateSSv1invertedE2/pinfty-page3} denote copies of $\pi_0(E_2)[[u_1,u_1^{-1}]]$). The charts of this height $2$ computations will show the shifting pattern of differentials.
    
    The first step is to compute $C_2$-Tate spectral sequence of $E_2/2^\infty$ via the cofiber sequence
    $$E_2 \longrightarrow 2^{-1}E_2 \longrightarrow E_2/2^\infty.$$
    The connecting homomorphism on the $\E_2$-page is an isomorphism and all the differentials follows from the classical geometric boundary theorem. The resulting differentials are shown in \Cref{fig:TateSSE2/pinfty-page3} and \Cref{fig:TateSSE2/pinfty-page7}. The spectral sequence vanishes from the $\E_8$-page.

\begin{figure}
\begin{center}
\makebox[\textwidth]{\includegraphics[trim={4cm 6cm 4cm 6cm}, clip, scale = 0.85]{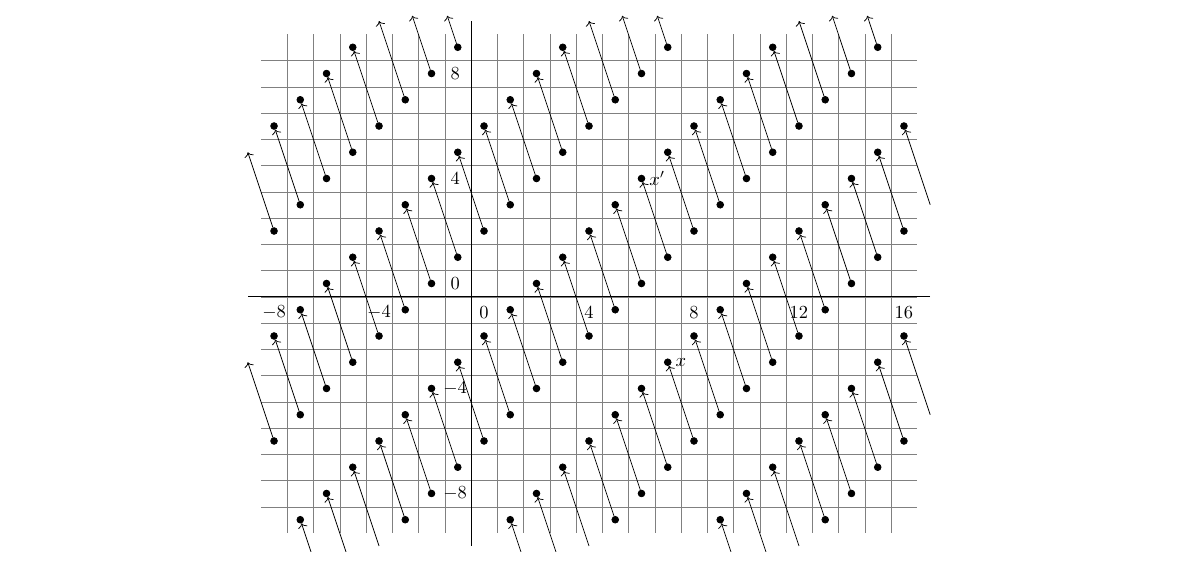}}
\end{center}
\begin{center}
\caption{The $\E_3$-page of $C_2$-Tate spectral sequence of $E_2/2^\infty$. The  $d_3$-differentials are multiplication by $u_1$ and have nontrivial cokernel.}
\hfill
\label{fig:TateSSE2/pinfty-page3}
\end{center}
\end{figure}

\begin{figure}
\begin{center}
\makebox[\textwidth]{\includegraphics[trim={4cm 6cm 4cm 6cm}, clip, scale = 0.85]{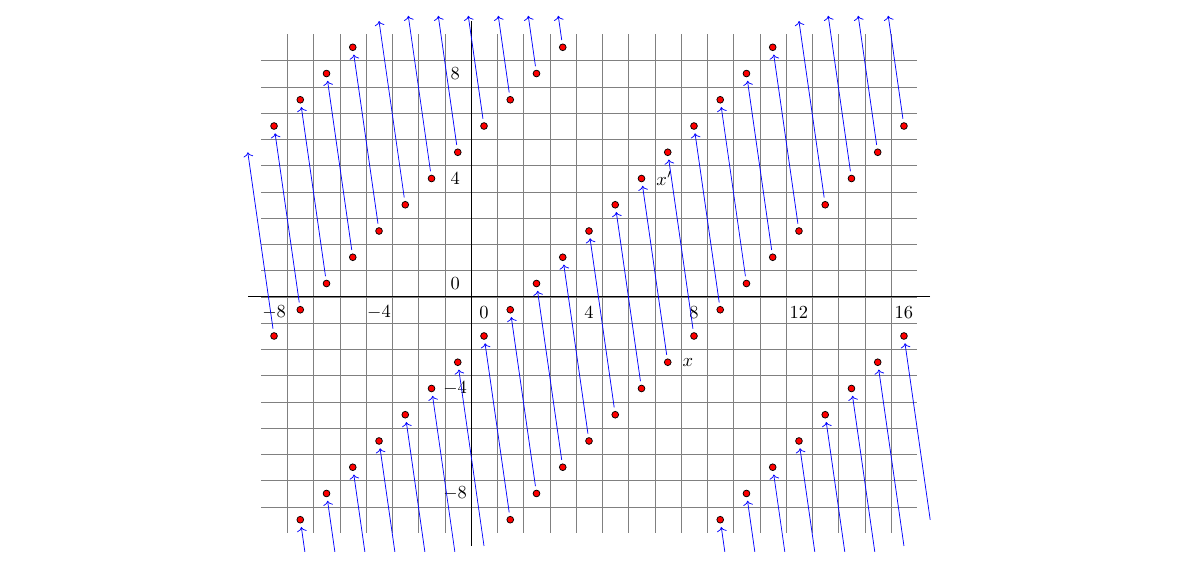}}
\end{center}
\begin{center}
\caption{The $\E_7$-page of $C_2$-Tate spectral sequence of $E_2/p^\infty$. The $d_7$-differentials are multiplication by $u_2$ and are isomorphisms. In particular, the $\E_8$-page vanishes.}
\hfill
\label{fig:TateSSE2/pinfty-page7}
\end{center}
\end{figure}

The next step is to compute the $C_2$-Tate spectral sequence of $E_2/(2^\infty,u_1^\infty)$ via the cofiber sequence
    $$E_2/2^\infty \longrightarrow {u_1}^{-1}E_2/2^\infty \longrightarrow E_2/(2^\infty,u_1^\infty).$$
    The $C_2$-Tate spectral sequence of the first term ($E_2/p^\infty$) is computed in the first step, and the $C_2$-Tate spectral sequence of the second term (${v_1}^{-1}E_2/2^\infty$) follows from the naturality of the spectral sequence map induced by $E_2/2^\infty \rightarrow {u_1}^{-1}E_2/2^\infty$. We show the result in \Cref{fig:TateSSv1invertedE2/pinfty-page3}; note that it vanishes from the $\E_4$-page (this is because $d_3$ is multiplication by $u_1$, and $u_1$ becomes a unit in ${u_1}^{-1}E_2/2^\infty$).
        
\begin{figure}
\begin{center}
\makebox[\textwidth]{\includegraphics[trim={4cm 6cm 4cm 6cm}, clip, scale = 0.85]{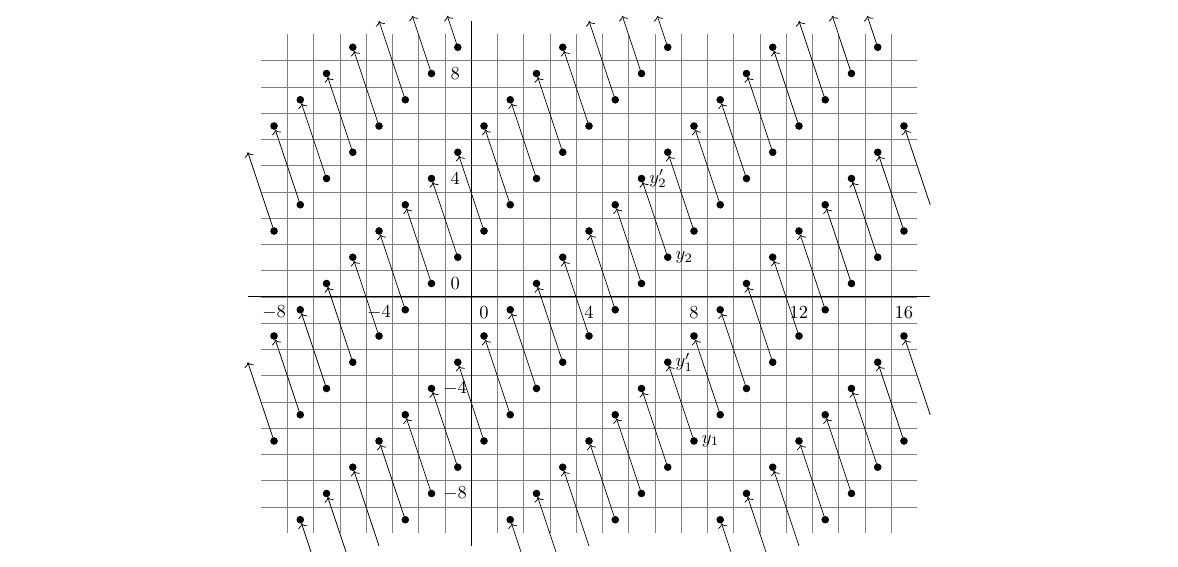}}
\end{center}
\begin{center}
\caption{The $\E_3$-page of $C_2$-Tate spectral sequence of $u_1^{-1}E_2/p^\infty$. The $d_3$-differentials are multiplication by $u_1$ and are an isomorphisms. In particular, the $\E_4$-page vanishes.}
\hfill
\label{fig:TateSSv1invertedE2/pinfty-page3}
\end{center}
\end{figure}

From the naturality of the spectral sequence map induced by ${u_1}^{-1}E_2/2^\infty \rightarrow E_2/(2^\infty,u_1^\infty)$, we have the $d_3$-differential as claimed in \Cref{inductive}, which is shown in \Cref{fig:TateSSE2/pinftyv1infty-page3}. However, neither the naturality nor the connecting morphism on the $\E_2$-page gives the desired $d_7$. Here we need to apply \Cref{geometricboundary} for $r=3$ to prove the desired $d_7$. We show the result in \Cref{fig:TateSSE2/pinftyv1infty-page7}. We mark the corresponding $x$, $x'$ in \Cref{fig:TateSSE2/pinfty-page3} and \Cref{fig:TateSSE2/pinfty-page7}, $y_1$, $y'_1$, $y_2$, $y'_2$ in \Cref{fig:TateSSv1invertedE2/pinfty-page3}, and $z$, $z'$ in \Cref{fig:TateSSE2/pinftyv1infty-page3} and \Cref{fig:TateSSE2/pinftyv1infty-page7}.

\begin{figure}
\begin{center}
\makebox[\textwidth]{\includegraphics[trim={4cm 6cm 4cm 6cm}, clip, scale = 0.85]{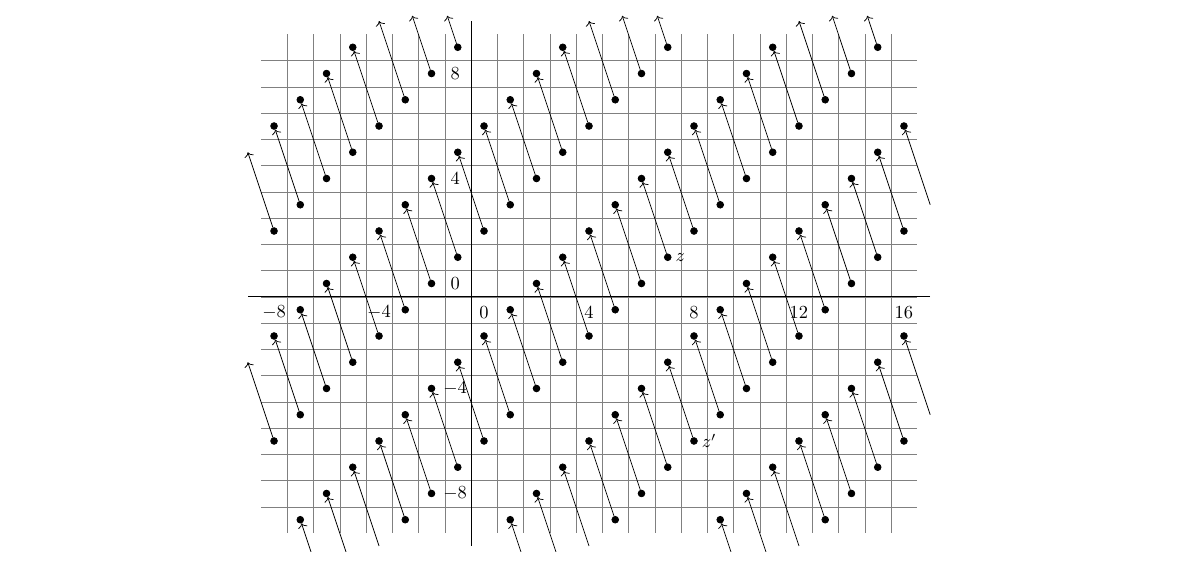}}
\end{center}
\begin{center}
\caption{The $\E_3$-page of $C_2$-Tate spectral sequence of $E_2/I_2^\infty$. The $d_3$-differentials are multiplication by $u_1$ and have nontrivial kernels.}
\hfill
\label{fig:TateSSE2/pinftyv1infty-page3}
\end{center}
\end{figure}

\begin{figure}
\begin{center}
\makebox[\textwidth]{\includegraphics[trim={4cm 6cm 4cm 6cm}, clip, scale = 0.85]{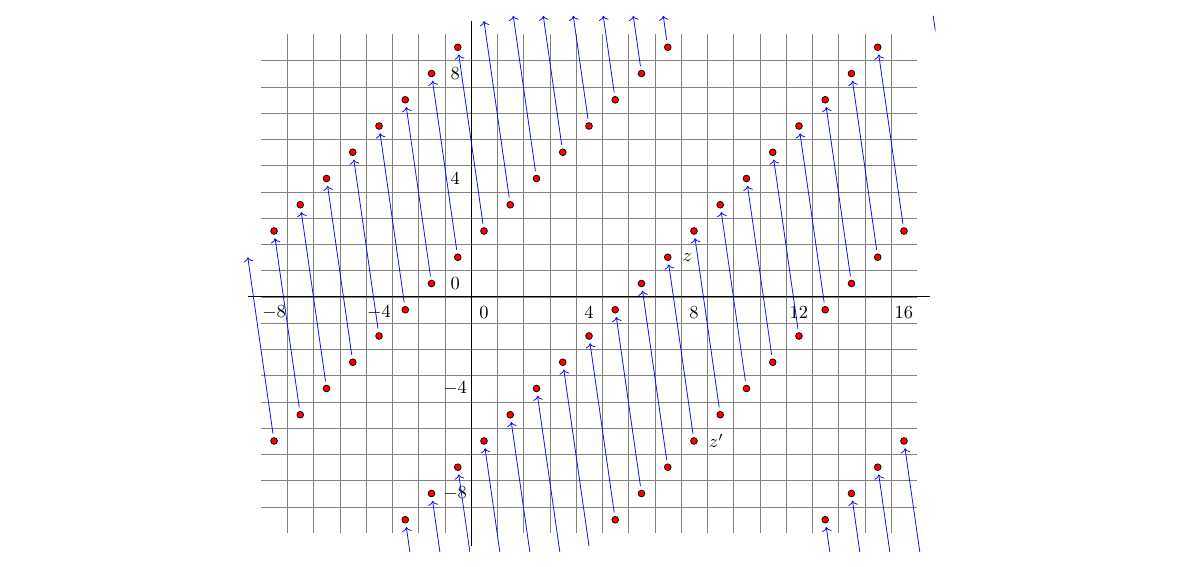}}
\end{center}
\begin{center}
\caption{The $\E_7$-page of $C_2$-Tate spectral sequence of $E_2/I_2^\infty$. The $d_7$-differentials are multiplying by $u_2$ and are isomorphisms. In particular, the $\E_8$-page vanishes.}
\hfill
\label{fig:TateSSE2/pinftyv1infty-page7}
\end{center}
\end{figure}
\end{ex}

\subsection{The homotopy orbit spectral sequence for \texorpdfstring{$(M_nE_n)_{hC_2}$}{MnEnhC2}}  
 The description of the Tate cohomology groups immediately implies an identification of the $\E_2$-term of the homotopy orbit spectral sequence via the isomorphism (not as rings)
\begin{equation}\label{equ:Tatecohomology}
\widehat{H}^s(C_2,\pi_tX) \cong H_{-s-1}(C_2,\pi_tX)
\end{equation}
for $s \le -2$. 
  
Unfortunately, unlike the Tate spectral sequence, the homotopy fixed points spectral sequence is not multiplicative.  We will name elements on $\E_2$-page of  the homotopy orbit spectral sequence by their names on the $\E_2$-page of the Tate spectral sequence via the above isomorphisms. To emphasise that it is not multiplicative, we introduce the following notation to state the $\E_2$-page of the homotopy orbit spectral sequence as a bigraded $\mathbb{Z}$-module. Given an abelian group $A$ and a set $S$ of bigraded elements, we denote the bigraded $\mathbb{Z}$-module
$\underset{s \in S}{\oplus} A$
by $AS$. This is the bigraded $\mathbb{Z}$-module such that for all elements $s$ in the set $S$, there is a copy of $A$ at the bidegree of $s$. In our case,
\begin{align*}
A&=((E_n)_0/I_n^\infty)\text{ or }\,(E_n/2)_0/(u_1^\infty,\cdots,u_{n-1}^\infty),\\
S&=\left\{\alpha^{-1}u^{2j+1} \mid j\in \Z\right\}\text{ or }\, \left\{\alpha^su^{2j+1} \mid \text{for }s<-1, j\in \Z\right\}.
\end{align*}

\begin{prop}\label{prop:E2HOSS}
The homotopy orbit spectral sequence for $E_n/I_n^\infty$ has $\E_2$-page
\begin{align*}
\E^2_{0,*}&=((E_n)_0/I_n^\infty)\left\{\alpha^{-1}u^{2j+1} \mid j\in \Z\right\}\\
\E^2_{s,*}&=(E_n/2)_0/(u_1^\infty,\cdots,u_{n-1}^\infty)\left\{\alpha^{-s-1}u^{2j+1} \mid \text{for }s<-1, j\in \Z\right\}
\end{align*}
where the $(t+s,s)$-bidegree of $\alpha^su^{2j+1}$ is $(4j+s+1,-s-1)$. 
\end{prop} 
\begin{proof}
We can compute
\[
H_s(C_2, \pi_* (E_n/I^\infty_n)_t)=
\begin{cases}
E_n/I^\infty_n &\text{for }s=0\text{ and }4\,|\,t;\\
(E_n/2)_0/(u_1^\infty,\cdots,u_{n-1}^\infty)  &\text{for }s>0\text{ even and }4\,|\,t;\\
(E_n/2)_0/(u_1^\infty,\cdots,u_{n-1}^\infty)  &\text{for }s>0 \text{ odd, and }2\,|\,t \text{ but } 4 \nmid t;\\
0 &\text{else}.
\end{cases}
\]
We name elements in the homotopy orbit spectral sequence by the same names of their corresponding elements the Tate spectral sequence via the isomorphism \eqref{equ:Tatecohomology}. There is a reindexing process from the $(t-s,s)$ Tate spectral sequence grading to the $(t+s,s)$ homotopy orbit spectral sequence grading. Let $x$ be an element in $(E_n/2)_0/(u_1^\infty,\cdots,u_{n-1}^\infty)$.  The element $a^su^{2j+1}x \in \hat{H}^s(C_2,(E_n/I_n^\infty)_{2s+4j+2})$ corresponds to an element in $H^{-s-1}(C_2,(E_n/I_n^\infty)_{2s+4j+2})$ under the isomorphism \eqref{equ:Tatecohomology}. The $(t+s,s)$ grading is $(4j+1+s,-s-1)$. This gives the desired result.
\end{proof}
\begin{thm}\label{thm:HOSSMnEn}
      All the nonzero differentials in the homotopy orbit spectral sequence for $E_n/I_n^\infty$ are:
    \[
    \begin{split}d_{2^{l+1}-1} \left(\alpha^s u^{2^{l+1}j+2^{l+1}-1}yx_n(m)\right) &=  u_l u^{2^{l+1}j+2^l-1} \alpha^{s+2^{l+1}-1}yx_n(m),
     \end{split}
     \]
where $s<-2^{l+1}$, $1 \leqslant l \leqslant n$, $m>0$, and $y \in (E_n/2)_0$.
    \end{thm}
\begin{proof}
By \Cref{monocomputation}, we have the following differentials in the Tate spectral sequence:
    \[
    \begin{split}d_{2^{l+1}-1} \left(\alpha^s u^{2^{l+1}j+2^{l+1}-1}yx_n(m)\right) &=  u_l u^{2^{l+1}j+2^l-1} \alpha^{s+2^{l+1}-1}yx_n(m),\,\,\,s\in \Z,\,1 \leqslant l \leqslant n.
     \end{split}
     \]
When $s<-2^{l+1}$, both the source and the target have nonzero preimage in the homotopy orbit spectral sequence under the isomorphism map in this range. By naturality we obtain the differentials in the homotopy orbit spectral sequence. There is no room for any more nonzero differentials by degree reasons.
\end{proof}

\begin{rem}
From the result, we could organize the homotopy orbit spectral sequence as follows: let $R_1$ be the non unital ring $(E_n)_0/(2^\infty,u_1^\infty,\cdots,u_{n-1}^\infty)$ and $\mathfrak{m}$ be the ideal generated by $(1/4,1/8)$.  Note that $\mathfrak{m} = (1/4,1/8,1/16,\cdots)$. Then $R_1/\mathfrak{m}$ consists of all order $2$ elements in $R_1$, and is isomorphic to the non unital ring $(E_n/2)_0/(u_1^\infty,\cdots,u_{n-1}^\infty)$. Denote the bigraded $(t,s)$-ring $R_1/\mathfrak{m}[\alpha^{-1}]$ by $R$, where elements in $R_1$ have bidegree $(0,0)$ and $|\alpha^{-1}|=(-2,1)$. By the module structure on the $\E_2$-page of the Tate spectral sequence for $E_n/I_n^\infty$, the $\E_2$-page of homotopy orbit spectral sequence for $E_n/I_n^\infty$ is a free rank one $R[u^{\pm 2}]$-module on one generator in degree $(0,0)$.  More specifically, the $\E_2$-page of the homotopy orbit spectral sequence for $E_n/I_n^\infty$, as a $R[u^{\pm 2}]$-module, is $R[u^{\pm 2}]\{x\}$ (where $|x|=(0,0)$).  All the differentials are $R$-linear.
\end{rem}

Note that \Cref{thm:HOSSMnEn} also determines the homotopy orbit spectral sequence of $(M_nE_n)_{hC_2}$, as the two spectral sequences differ by a shift of $-n$ by \Cref{prop:equivmonoequiv}.  Since we have a complete description of the homotopy fixed point spectral sequence of $(IE_n)^{hC_2}$, we can apply \Cref{lem:qzdual} to give our second proof of \Cref{thm:grosshopkinshc2}. We will use the following lemma from \cite{bbs_gross}: 
\begin{lem}[Lemma~4.7 in \cite{bbs_gross}]\label{lem:ss}
Suppose $\mathcal{M_*^{*,*}}$ is a spectral sequence of modules over a spectral sequence $\mathcal{E}^{*,*}_*$ of algebras. Assume additionally that $\mathcal{M}_2^{*,*}$ is free of rank one on a class $x$, and $x$ survives to the $\cal{E}_\infty$-page, then the spectral sequence $\mathcal{M_*^{*,*}}$ is, up to a shift, isomorphic to the spectral sequence of $\mathcal{E}_*^{*,*}$. 
\end{lem}
\begin{thm}\label{thm:HFSSIEn}
The $\E_2$-page of the homotopy fixed points spectral sequence for $IE_n$, as a $E_n[\alpha, u^{\pm 2}]/2\alpha$-module, is free of rank one on a generator $\{x\}$ in $(n+4,0)$ with the $(t-s,s)$ grading.

All the differentials are $E_n[\alpha]/2\alpha$-linear and are determined by the differentials
$$d_{2^{l+1}-1}(u^{2^{l+1}m+2^l}x) = u_l \alpha^{u^{2^{l+1}m+2^{l+1}-1}}x,\,\,\,m\in\Z,\, 1 \leq l \leq n.$$

In particular, $x$ is permanent cycle and there is an equivalence $IE_n^{hC_2} \simeq \Sigma^{4+n}E_n^{hC_2}$.
\end{thm}
\begin{proof}
From \eqref{eq:detinen} we know that $C_2$ acts trivially on $\pi_{n+4k}IE_n$ and acts by $(-1)$ on $\pi_{n+4k+2}IE_n$ for $k \in \Z$. Therefore, the $\E_2$-page of the homotopy fixed points spectral sequence for $IE_n$ is a shift of the $\E_2$-page of the homotopy fixed points spectral sequence for $E_n$ by $n+4k$. Taking a generator $x \in \pi_{n+4}IE_n$, we can write the $\E_2$-page as $\left( E_n[\alpha, u^{\pm 2}]/2\alpha\right)\{x\}$. 

Note that HFPSS($IE_n^{hC_2}$) is a module over HFPSS($E_n^{hC_2}$). We will show that $x$ is a permanent cycle and then \Cref{lem:ss} will imply that HFPSS($IE_n^{hC_2}$) is a $(n+4)$-shift of HFPSS($E_n^{hC_2}$). By \Cref{prop:equivmonoequiv}, we know that the HOSS($(M_nE_n)_{hC_2}$) is a $n$-shift of the HOSS($(E_n/I_n^\infty)_{hC_2}$). We have computed the latter in \Cref{thm:HOSSMnEn}.  The differentials are 
    \[
    \begin{split}d_{2^{l+1}-1} \left(\alpha^s u^{2^{l+1}j+2^{l+1}-1}yx_n(m)\right) &=  u_l u^{2^{l+1}j+2^l-1} \alpha^{s+2^{l+1}-1}yx_n(m),
     \end{split}
     \]
where $s<-2^{l+1}$, $1 \leqslant l \leqslant n$, $m>0$, and $y \in (E_n/2)_0$.

Note that $|y|=|x_n(m)|=(0,0)$. The $(t+s,s)$ grading of the differentials hitting the $(t+s)$-axis (the horizontal axis) are
\begin{eqnarray*}
d_3 \colon (-n-7+8k,3)&\longrightarrow &(-n-8+8k,0)\\
d_7 \colon (-n-11+16k,7)&\longrightarrow &(-n-12+16k,0)\\
& \vdots&\\
d_2^{l+1}-1 \colon (-n-2^{l+1}-3+2^{l+2}k,2^{l+1}-1)&\longrightarrow &(-n-2^{l+1}-4+2^{l+2}k,0).
\end{eqnarray*}
In particular, there are no nonzero differentials hitting the bidegree $(-n-4,0)$. By \Cref{lem:qzdual}, the homotopy fixed points spectral sequence for $IE_n$ has differentials Pontryagin dual to the differentials in the homotopy orbit spectral sequence for $M_nE_n$.  Therefore, all differentials from classes in bidegree $(n+4,0)$ are $0$. In particular, $x$ is a permanent cycle.
By \Cref{lem:ss}, the HFPSS($IE_n^{hC_2}$) is a $(n+4)$-shift of the HFPSS($E_n^{hC_2}$). We have computed that $\pi_*IE_n^{hC_2} \cong \pi_*E_n^{hC_2}\{x\}$ as a $\pi_*E_n^{hC_2}$-module, where $x$ is in stem $(n+4)$. This implies that
    $$L_{K(n)}(E_n^{hC_2} \otimes S^{n+4}) \xrightarrow{L_{K(n)}(\id\otimes x)} L_{K(n)}(E_n^{hC_2} \otimes IE_n^{hC_2})\longrightarrow IE_n^{hC_2}$$
    is a weak equivalence (the last map uses the fact that $IE_n^{hC_2}$ is a $E_n^{hC_2}$-module). Therefore, we have the equivalence $IE_n^{hC_2} \simeq \Sigma^{4+n}E_n^{hC_2}$.   
    \end{proof}
    \begin{rem}
      Suppose that $F \subset \G_n$ is a finite subgroup of the Morava stabilizer group, and let $F_0 = F \cap \mathbb{S}_n$. Suppose that the canonical map $F/F_0 \to \mathbb{G}_n/\mathbb{S}_n$ is an isomorphism, and that $F_0$ has $2$-Sylow subgroup equal to $C_2$, then we claim that $I_nE_n^{hF_0} \simeq \Sigma^{4+n}E_n^{hF_0}$ as well, and similarly for $E_n^{hF}$. We consider the case of $F_0$ first. There is a spectral sequence
\[
H^s(F_0,\pi_t(IE_n)) \implies \pi_{t-s}IE_n^{hF_0}.
\]
The assumption on $F_0$ implies that the Lyndon--Hochschild--Serre spectral sequence associated to the extension $1 \to C_2 \to F_0 \to \F_0/C_2 \to 1$  collapses to give an isomorphism
\[
H^{*}(F_0,\pi_t(IE_n)) \cong H^*(C_2;\pi_t(IE_n))^{F_0/C_2}.
\]
Using the map of spectral sequences obtain from the map $IE_n^{hF_0} \to IE_n^{hC_2}$ and the above isomorphism, one sees that the spectral sequence for $IE_n^{hF_0}$ is isomorphic to the spectral sequence of $E_n^{hF_0}$ up to a shift by $n+4$. An argument similar to \Cref{thm:HFSSIEn} shows that this implies that $IE_n^{hF_0} \simeq E_n^{hF_0}$. 

The argument for $E_n^{hF}$ is very similar; we just need to know that the spectral sequence for $F$ can be obtained from that of $F_0$ by taking Galois invariants. The assumption on $F/F_0$ along with \cite[Lemma 1.32]{bobkova_goerss} show that this holds, and we are done. 
    \end{rem}
\section{The exotic part of the \texorpdfstring{$K(n)$}{K(n)}-local Picard group}\label{sec:exoticpic}
In this section, we outline one reason for being interested in the calculation of the Gross--Hopkins dual of $E_n^{hC_2}$, which is to prove that, when combined with work of Beaudry, Goerss, Hopkins, and Stojanoska \cite{bghs} and Barthel, Beaudry, Goerss, and Stojanoska \cite{bbgs}, we can see that there are non-trivial 'exotic' elements in the $K(n)$-local Picard group. 

\begin{rem}
    Here we rely strongly on forthcoming work of Beaudry, Goerss, Hopkins, and Stojanoska \cite{bghs}. Of course, the cautious reader may therefore want to consider \Cref{prop:exotic} below as conditional until this result has appeared.   
 \end{rem}
\begin{rem}
  In this section, we work $K(n)$-locally, so that all spectra are $K(n)$-localized, and $M \otimes N = L_{K(n)}(M \otimes N)$. 
\end{rem}
Let $\Pic_n$ denote the Picard group of $K(n)$-locally invertible spectra. A remarkable result, proved in \cite{hms_pic1}, is that a $K(n)$-local spectrum $X$ is in $\Pic_n$ if and only if $\mE X \cong (E_n)_*$ as $(E_n)_*$-modules (up to suspension).  We can then consider the subgroup $\kappa_n \subset \Pic_n$ of those invertible spectra for which there is an isomorphism of \emph{Morava modules} $\mE X \cong (E_n)_*$. An argument with the $K(n)$-local Adams spectral sequences shows that whenever $p \gg n$, $\kappa_n$ is trivial - if $\mE X \cong (E_n)_*$ as Morava modules, then we must have $X \simeq L_{K(n)}S^0$. 

 In other cases however, it is possible that $\kappa_n \ne 0$. For example, when $n=1$ and $p = 2$, then $\kappa_1 \cong \Z/2$, generated by $\Sigma^2L_{K(1)}DQ$, where $L_{K(1)}DQ$ is the $K(1)$-localization of the dual of the question-mark complex, see \cite[p.~650]{goerss_henn}. In fact, using the calculation of $I_1E_1^{hC_2}$, we can prove that there must be \emph{some} non-trivial element in $\kappa_1$ (of course, this does not identify it explicitly). In order to explain this, we recall that $\mE I \cong \Sigma^{n^2-n} \mE \langle \det \rangle$ as Morava modules. When $n = 1$, this simplifies to the statement that $\mE I \cong \Sigma^2 (E_n)_*$. It follows that we must have that $I \simeq S^2 \otimes P$, where $P \in \kappa_1$ (recall that we work $K(1)$-locally, so $S^2$ really means $L_{K(1)}S^2$ and $\otimes$ refers to the $K(1)$-local tensor product).  The computation of $IE_1^{hC_2}$ implies that $\kappa_1 \ne 0$, as we now explain. 
\begin{ex}\label{ex:n=1}
  The $n=1$ version of our main theorem is that $IE_1^{hC_2} \simeq \Sigma^5 E_1^{hC_2}$. Since $I$ is $K(n)$-locally invertible, it is in particular dualizable in the $K(n)$-local category, so that $IE_1^{hC_2} \simeq \Sigma^5E_1^{hC_2} \simeq I \otimes DE_1^{hC_2}$, where $DE_1^{hC_2}$ denotes the $K(1)$-local dual of $E_1^{hC_2}$. The latter is known to be equivalent to $\Sigma^{-1}E_1^{hC_2}$, see \cite[Lem.~8.16]{hahn_mitchell}, so that $I \otimes E_1^{hC_2} \simeq \Sigma^6E_1^{hC_2}$. 

  Now suppose that there were no exotic elements in the $K(1)$-local Picard group when $n=1,p=2$, so that we would have $I \simeq S^2$. This would imply that $I \otimes E_1^{hC_2} \simeq \Sigma^2E_1^{hC_2}$, which contradicts the calculation in the previous paragraph because $E_1^{hC_2}$ is 8-periodic. We see that we must have $I \simeq \Sigma^2 X$, where $X \in \kappa_1$ is an exotic element in the $K(1)$-local Picard group such that $X \otimes E_1^{hC_2} \simeq \Sigma^4E_1^{hC_2}$; indeed, $X \simeq L_{K(1)}DQ$ has this property.
\end{ex}
Our calculation of $IE_n^{hC_2}$ is the first step in generalizing these methods to higher heights. The problem of determining the duals $DE_n^{hF}$ for finite subgroups $F \subset \G_n$ is under investigation by Beaudry, Goerss, Hopkins, and Stojanoska \cite{bghs}.\footnote{The interested reader can find a talk by Hopkins online at \url{https://youtu.be/Ix4pg87LKVk}.} In the case of $F = C_2$ they prove the following (which is known when $n = 1$, see \Cref{ex:n=1} above). 
\begin{thm}[Beaudry--Goerss--Hopkins--Stojanoska]\label{conj1}
  The $K(n)$-local dual $DE_n^{hC_2} \simeq \Sigma^{-n^2}E_n^{hC_2}$. 
\end{thm}

When $p$ is odd it is well-known how to produce a spectrum $S \langle \det \rangle$ such that $(E_n)_*S\langle \det \rangle \cong (E_n)_*\langle \det \rangle$ see e.g., \cite{goerss_henn}. In the case $p=2$ such a spectrum has been constructed by Barthel, Beaudry, Goerss, and Stojanoska \cite{bbgs} with the following properties. 
\begin{thm}[Barthel--Beaudry--Goerss--Stojanoska]\label{conj2}
  There is a spectrum $S \langle \det \rangle$ such that $\mE S \langle \det \rangle \cong (E_n)_*\langle \det \rangle$. Moreover, we have $E_n^{hC_2} \otimes S \langle \det \rangle \simeq \Sigma^{1-(-1)^n}E_n^{hC_2}$. 
\end{thm}
\begin{proof}
  The spectrum $S \langle \det \rangle$ is constructed in \cite[Def.~3.2]{bbgs} and the Morava module structure is proved in \cite[Thm.~3.1]{bbgs}. In the case that $n$ is even, $C_2$ is in the kernel of the determinant, and so \cite[Cor.~3.1]{bbgs} shows that $E_n^{hC_2} \otimes S \langle \det \rangle \simeq E_n^{hC_2}$. 

  When $n$ is odd, $C_2$ is not in the kernel of the determinant, and to calculate $E_n^{hC_2} \otimes S \langle \det \rangle$, we need to understand the construction of $S \langle \det \rangle$. The construction relies on a $\G_n$-equivariant spectrum $S(1)$, the Tate sphere, whose underlying spectrum is the $p$-complete sphere. This is equipped with a continuous action of $\Z_p^{\times}$, which is extended to an action of $\mathbb{G}_n$ via the determinant map $\det \colon \G_n \to \Z_p^{\times}$. The determinantal sphere is then defined as 
\[
S \langle \det \rangle  = (E_n \otimes S(1))^{h\mathbb{G}_n},
\]
where the right-hand side has the diagonal action. 

Because $S\langle \det \rangle$ is dualizable and has trivial $C_2$ action, we have
\[
E_n^{hC_2} \otimes S\langle \det \rangle \simeq (E_n \otimes S\langle \det \rangle)^{hC_2}.
\]
By \cite[Thm.~3.1]{bbgs}, there is a canonical $\mathbb{G}_n$-equivariant equivalence
\[
f \colon E_n \otimes S\langle \det \rangle \rightarrow E_n \otimes S(1),
\]
and hence
\[
(E_n \otimes S\langle \det \rangle)^{hC_2} \simeq (E_n \otimes S(1))^{hC_2}.
\]
When $n$ is odd, $C_2 \subset \mathbb{G}_n$ maps isomorphically to $\{ \pm 1 \} \subset \Z_2^{\times}$ under the determinant map, and this determines the $C_2$-action on $S(1)$. In the proof of \cite[Prop.~4.3]{bbgs} the authors show that, $C_2$-equivariantly, $S(1)$ is the cofiber of the transfer $S^0 \to \Sigma^{\infty}_+C_2$. By taking the Spanier--Whitehead dual of the cofiber sequence 
\[
\Sigma^{\infty}_+C_2 \to S^0 \to S^{\sigma}
\]
we see that, as a $C_2$-spectrum, this gives an equivalence $S(1) \simeq S^{1 -\sigma}$. It follows that
\[
E_n^{hC_2} \otimes S\langle \det \rangle \simeq (E_n \otimes S(1))^{hC_2} \simeq (\Sigma^2E_n)^{hC_2} \simeq \Sigma^2 E_n^{hC_2},
\] 
where we have used that $E_n$ is $\rho = 1 + \sigma$-periodic as a $C_2$-spectrum. 
\end{proof}
With these two results and our main result, we easily see that $\kappa_n \ne 0$.

\begin{thm}\label{prop:exotic}
When $p=2$, $\kappa_n \ne 0$ and has an element whose order is at least $2^{n+2-v_2(2n+3+(-1)^n)}$ where $v_2(-)$ is the 2-adic valuation.
 \end{thm} 
\begin{proof}
   This is similar to \Cref{ex:n=1}. We know $I \simeq \Sigma^{n^2-n}S \langle\det\rangle \otimes P$ where $P \in \kappa_n$. Since $IE_n^{hC_2} \simeq I \otimes DE_n^{hC_2} \simeq \Sigma^{4+n}E_n^{hC_2} $, using \Cref{conj1} we deduce that $I \otimes E_n^{hC_2} \simeq \Sigma^{4+n+n^2}E_n^{hC_2}$. On the other hand, \Cref{conj2} implies that $\Sigma^{n^2-n}S \langle \det \rangle \otimes E_n^{hC_2} \simeq \Sigma^{n^2-n+1-(-1)^n}E_n^{hC_2}$.  This implies that $P \otimes E_n^{hC_2} \simeq \Sigma^{2n+3+(-1)^n}E_n^{hC_2}$. The periodicity of $E_n^{hC_2}$ is $2^{n+2}$ and $2n+3+(-1)^n \neq 0$ in $\Z/2^{n+2}$. Hence, the element $P \in \kappa_n$ is nontrivial. In particular, if $P$ has finite order $k$, then $P^k \otimes E_n^{hC_2}=\Sigma^{(2n+3+(-1)^n)k}E_n^{hC_2} \simeq E_n^{hC_2}$. So $2^{n+2}|(2n+3+(-1)^n)k$ and $k$ is at least $2^{n+2-v_2(2n+3+(-1)^n)}$.
\end{proof}
\begin{rem}
If one view $\kappa_n$ as a measure how non-algebraic the $K(n)$-local category is, the previous result shows that at prime $2$, the size of the non-algebraic part is growing at least exponentially with respect to the height $n$ (note that $\kappa_n$ could have infinite order as well). The fast growth has been seen in examples at height 1 and 2, where $|\kappa_1|=2$ and $|\kappa_2|=2^9$ (as computed in forthcoming work of Beaudry--Bobkova--Goerss--Henn).
\end{rem}
    \appendix

\section{Generalized geometric boundary theorem}\label{app:gbt}

To compute the homotopy fixed points spectral sequence of $E_n/I_n^\infty$ (\Cref{monocomputation}), we will inductively compute the homotopy fixed points spectral sequence of $E_n/I_k^\infty$ for $k \leqslant n$, see \Cref{inductive}. During our inductive step, we will make precise and prove the following claim, which is a tailored version of case 3 in the Generalized Geometric Boundary Theorem of Behrens \cite[Lem.~A.4.1]{EHP}. Our proof uses universal models, which is different from Behrens' proof based on diagram chasing.

Let $X \overset{i}{\rightarrow} Y \overset{p}{\rightarrow} Z$ be a fiber sequence with compatible filtration towers (this will be made precise in the theorem below).  In the spectral sequences associated with these towers, if there are non-trivial differentials 
$$\,d_s x= x', \, d_r y_1 = y_1', \, d_r y_2=y_2'$$ 
and 
$$i(x)=y_1', \,i(x')=y_2', \, p(y_1)=z, \,p(y_2)=z',$$ 
then $z'$ is non-zero in $E_{r+1}$ and there is a non-trivial differential 
$$d_s z= z'.$$

The following diagram provides a visualization of the situation: 
$$
   \setlength\mathsurround{0pt}
\begin{tikzcd}
x'  \arrow[rr, "i"] & & y'_2 & & & \\
& & & y_2  \arrow[ul, swap, "d_r"] \arrow[r, "p"] & z' & \\
& x \ar[uul, "d_s"] \arrow[r, "i"] & y'_1 & & &  \\
& & & y_1 \arrow[ul, swap, "d_r"] \arrow[rr, "p"] & & z \arrow[dashrightarrow, "d_s", uul]
\end{tikzcd}
$$

We will prove this version of the Generalized Geometric Boundary Theorem using the universal model method for differentials.  Note that our proof is different from the proof in \cite{EHP}, which requires extensive diagram chasing.  Throughout our formulation and proof of the claim, we will work in the category of bounded below towers. 
\begin{defn} \rm
Denote the category of towers indexed over nonnegative integers by $\mathcal{T}$.  An object in $\mathcal{T}$ is a tower
$$X_\bullet=\{ X_0 \overset{\alpha_0}{\leftarrow} X_1\overset{\alpha_1}{\leftarrow} X_2 \overset{\alpha_2}{\leftarrow} \cdots \}$$
of spectra. Here, $X_i$'s are spectra and the $\alpha_k$'s are morphisms of spectra.  A morphism 
$$f_\bullet = (f_0, f_1, \dotsc) \colon X_\bullet \rightarrow Y_\bullet$$ 
in $\mathcal{T}$ is a collection of morphisms $f_i \colon X_i \rightarrow Y_i$ that satisfies the following commutative diagram:

    $$
       \setlength\mathsurround{0pt}
    \begin{tikzcd}
 X_0 \arrow{d}{f_0} \arrow[leftarrow]{r}{\alpha_0} & X_1 \arrow[leftarrow]{r}{\alpha_1} \arrow{d}{f_1} & X_2 \arrow[leftarrow]{r}{\alpha_2} \arrow{d}{f_2}& \cdots \\
Y_0 \arrow[leftarrow]{r}{\alpha_0'} & Y_1 \arrow[leftarrow]{r}{\alpha_1'} & Y_2 \arrow[leftarrow]{r}{\alpha_2'} & \cdots.
    \end{tikzcd}
    $$
\end{defn}

\begin{defn} \rm
Given towers $X_\bullet, Y_\bullet, Z_\bullet \in \mathcal{T}$, a sequence $X_\bullet \rightarrow Y_\bullet \rightarrow Z_\bullet$ in $\mathcal{T}$ is a \textit{fiber sequence} if the level wise maps $X_k \rightarrow Y_k  \rightarrow Z_k$ are fiber sequences in the category of spectra for all $k \geq 0$.
\end{defn}

Let the cofiber of $X_{i+1} \rightarrow X_i$ be $X_i/X_{i+1}$, and the spectral sequence associated to the tower $X_\bullet \in \mathcal{T}$ be $SS(X)$, with the $r$\textsuperscript{th} page $E^{s,t}_r(X)$.  Then
$$E_1^{s,t}(X)=\pi_t(X_s/X_{s+1}) \Rightarrow \pi_{t-s}X.$$
This is an Adams type spectral sequence.  A morphism $f_\bullet \colon X \rightarrow Y$ in $\mathcal{T}$ induces a map $f_* \colon SS(X) \rightarrow SS(Y)$ of the associated spectral sequences.  Denote the induced map 
$$X_r/X_{r+1} \rightarrow Y_r/Y_{r+1}$$ 
on the cofibers by $f_{r/r+1}$.

We will introduce a universal model for a $d_r$-differential as an object in $\mathcal{T}$.  We will see that in the spectral sequence associated to this object, there is a universal $d_r$-differential. 

Consider the tower $D(s,r) = (D(s,r)_0 , D(s,r)_1, \dotsc) \in \mathcal{T}$ given by

    $$\begin{cases}
    D(s,r)_0 = \dotsb = D(s,r)_s = D^1,\\
    D(s,r)_{s+1} = \dotsb = D(s,r)_{s+r} = S^0,\\
    D(s,r)_{k} = * ~ for ~k>s+r.
    \end{cases}
    $$
The tower corresponding to $D(s,r)$ is
    $$
       \setlength\mathsurround{0pt}
    \begin{tikzcd}
      D^1 \arrow[hookleftarrow]{r} \arrow{d} & S^0 \arrow[leftarrow]{r}{\operatorname{id}} & \cdots \arrow[leftarrow]{r}{\operatorname{id}} & S^0 \arrow{d}\arrow[leftarrow]{r} & *\\
      S^1 \arrow[dashed]{ur} & & & S^0 &
    \end{tikzcd}
    $$
In the spectral sequence $SS(D(s,r))$,  we have 
$$E_1^{s,1}=\pi_1(D(s,r)_r/D(s,r)_{r+1})=\pi_1(D^1/S^0)=\pi_1S^1$$ 
and 
$$E_1^{s+r,0}=\pi_0(D(s,r)_{s+r}/D(s,r)_{s+r+1})=\pi_0(S^0/*)=\pi_0S^0.$$ 
Let $a$ be the generator of $E_1^{s,1}$ and $b$ be the generator of $E_1^{s+r,0}$ such that $$d_r(a)=b$$ 
in $SS(D(s,r))$.  This differential is our universal $d_r$-differential.  It has the following universal property.

\begin{prop}[Universal property]\label{prop:drUniversalProperty}
Let $X_\bullet \in \mathcal{T}$.  Suppose there is a nontrivial $d_r$-differential 
$$d_r x = x'$$
in $SS(X_\bullet)$.  Then there is a morphism 
$$\Sigma^t D(s,r) \overset{f_\bullet}{\rightarrow} X_\bullet$$ 
in $\mathcal{T}$ such that $f_*(a)=x$ and $f_*(b)=x'$.
\end{prop}
\begin{proof}
This follows directly from our construction of $D(s,r)$.  
\end{proof}


\begin{rem}\rm

In \Cref{prop:drUniversalProperty}, the choice of the map $f \colon \Sigma^t D(s,r) \rightarrow X_\bullet$ is not unique.  More specifically, the map is determined by a choice of lifts of $x$ to $\tilde{x} \in \pi_{t-1}X_{s+r}$ and of $x'$ to $\tilde{x}' \in \pi_t(X_s/X_{s+r})$.  By a standard procedure, we can find representatives $(x_1, x_1')$ of $(x, x')$ on the $E_1$-page. Given these representatives, there is a bijection between lifts $\tilde{x}$ of $x_1$ and 
$$\operatorname{im}(\pi_t(X_{s+1}/X_{s+r}) \rightarrow \pi_t(X_s/X_{s+r})).$$ 
Similarly, there is a bijection between lifts $\tilde{x}'$ of $x_1'$ with 
$$\operatorname{im}(\pi_t(X_{s+r+1}) \rightarrow \pi_t(X_{s+r})) \cap \operatorname{ker}(\pi_t X_{s+r} \rightarrow \pi_tX_s).$$
This provides an algebraic description of the choices of $f$. 
\end{rem}

\begin{thm}[Generalized Geometric Boundary Theorem]{\label{geometricboundary}}

  Let $X_\bullet \xrightarrow{i} Y_\bullet \xrightarrow{p} Z_\bullet$ be a fiber sequence in $\mathcal{T}$.  Suppose there are nontrivial differentials 
\begin{eqnarray*}
d_{r'}x&=&x' \,\,\, (\text{in } SS(X)), \\
d_r y_1 &=& y_1'\,\,\, (\text{in }SS(Y)), \\
d_r y_2&=&y_2' \,\,\,(\text{in }SS(Y)),
\end{eqnarray*}
and 
\begin{eqnarray*}
i(x)&=&y_1', \\
i(x')&=&y_2', \\ 
p(y_1)&=&z, \\
p(y_2)&=&z'. 
\end{eqnarray*}
Then there is a nontrivial differential $d_{r'} z= z'$ in $SS(Z)$.

\end{thm}

  \begin{rem}\rm
    When $r=1$, \Cref{geometricboundary} is the classical Geometric Boundary Theorem (see \cite{bruner} and \cite[Pro.~2.3.4 and Lem.~2.3.7]{greenbook}).  The most generalized version has been studied by Behrens, see \cite[Lem.~A.4.1]{EHP}.  In particular, \Cref{geometricboundary} is Case 3 of Behrens's generalized version, however the proof we present is different from that given by Behrens. 
  \end{rem}

\begin{proof}[Proof of \Cref{geometricboundary}] We will construct a map
$$u_\bullet \colon D_X \longrightarrow D_Y$$
which captures the conditions of the theorem.  For the simplicity, we will assume that the $(s,t)$-degree of $x$ is $(s,1)$ (if $|x|=(s,t)$, we can just suspend our map by $t$).
    
Let $D_1 \in \mathcal{T}$ be the constant tower with $(D_{1})_*=D^1$ for all $*$, and all maps the identity maps.  Let $D_X$ be the tower $D(s+r,r')$ that captures the differential
$$d_{r'}x=x'.$$
More precisely, $D_x$ is the tower such that there is a map 
$$f_\bullet \colon D_x \longrightarrow X_\bullet$$ 
with $f_*(a_X)=x$, $f_*(b_X)=x'$ and $d_r(a_X)=b_X$.  Define the tower $D_Y$ as follows:
$$(D_{Y})_*=\left\{\begin{array}{ll}
               D(s,r)_{*}  & {*} \leqslant s+r\\
               (D_{1})_{*} & s+r< {*} \leqslant r'\\
               {D(s+r',r)}_{*} & r'<{*}
            \end{array}\right.$$
The maps are all identity maps except for the maps
$$(D_{Y})_{r+1} \hookrightarrow (D_{Y})_r,$$ 
which is $i^+ \colon D^1  \hookrightarrow S^1$, the embedding of $D^1$ into the upper hemisphere and
$$(D_{Y})_{r'+1} \hookrightarrow (D_{Y})_{r'},$$
which is $S^0 \hookrightarrow D^1$, the embedding of the boundary.  Denote the universal differential in the $D(s,r)$ part as 
$$d_r (a_{Y,1})=b_{Y,1}$$ 
and the universal differential in the $D(s+r',r)$ part as 
$$d_r (a_{Y,2})=b_{Y,2}.$$

The map $u_\bullet \colon D_X \rightarrow D_Y$ is defined as follows:
\[
  u_*=\begin{cases}
               * \leqslant s \colon D^1 \xrightarrow{i^{-}} S^1 \hookrightarrow D^2 ~\mbox{($i^-$ is the embedding into the bottom hemishpere)}\\
               s<*\leqslant s+r ~ \colon D^1 \xrightarrow{i^{-}} S^1\\
               s+r<*\leqslant s+r' \colon S^0 \hookrightarrow D^1 \\
               s+r'<*\leqslant s+r+r' \colon S^0 \xrightarrow{id} S^0 \\
               s+r+r' <*  \colon * \xrightarrow{id} *
            \end{cases}
\]

Given this map, one can directly check that the level-wise cofiber of $u_\bullet$ is $D(s,r')$.  We will denote the universal differential in this $D(s,r')$ by 
$$d_{r'}(a_Z)=b_Z.$$

We will show that there exist maps $f_\bullet \colon D_X \rightarrow X$ and $g_\bullet \colon D_Y \rightarrow Y$ such that the diagram
$$
   \setlength\mathsurround{0pt}
    \begin{tikzcd}
      D_X \arrow[r, "f_\bullet"] \arrow[d] & X \arrow[d] \\
      D_Y \arrow[r, "g_\bullet"] \arrow[d]& Y \arrow[d]\\
      D(s,r') \arrow[r,"h_\bullet",dashed] & Z.
    \end{tikzcd}
    $$
is commutative and the following relations hold:
\begin{eqnarray*}
f_*(a_X)&=&x, \\
f_*(b_X)&=&x', \\
g_*(a_{Y,1})&=&y_1, \\
g_*(b_{Y,1})&=&y'_1, \\
g_*(a_{Y,2})&=&y_2, \\
g_*(b_{Y,2})&=&y'_2. 
\end{eqnarray*}
Once we have shown this, the induced map 
$$h_\bullet \colon D(s,r')  \rightarrow Z_\bullet$$
on the cofibers will send the universal differential $d_{r'}(a_Z)=b_Z$ to the desired differential $d_{r'}(z)=z'$.
  
The map $f_\bullet$ can be easily constructed from the universal property of $D(s,r')$.  To construct the map $g_\bullet$, we need to construct the following maps: 
\begin{eqnarray*}
g_{s+r+r'}\colon&& (D_Y)_{s+r+r'}=S^0 \rightarrow Y_{s+r+r'}, \\
g_{s+r'} \colon&& (D_Y)_{s+r'}=D^1 \rightarrow Y_{s+r'},\\
g_{s+r} \colon&& (D_Y)_{s+r}=S^1 \rightarrow Y_{s+r},\\
g_{s} \colon&& (D_Y)_{s}=D^2 \rightarrow Y_{s}.
\end{eqnarray*}
    
By an abuse of notation, we will denoted the above four maps by $g_1,\cdots, g_4$, respectively.  We will construct $g_1,\cdots, g_4$ such that the following diagram is commutative: 
    $$
       \setlength\mathsurround{0pt}
    \begin{tikzcd}
    S^0 \arrow[r,"g_1"] \arrow[d,hookrightarrow] & Y_{s+r+r'} \arrow[d]\\
    D^{1} \arrow[r,"g_2"] \arrow[d,"i^+"] & Y_{s+r'} \arrow[d]\\
    S^{1} \arrow[r,"g_3"] \arrow[d,hookrightarrow]& Y_{s+r} \arrow[d]\\
    D^{2} \arrow[r,"g_4"] & Y_s.
    \end{tikzcd}
    $$
    
    The map $g_1$ can be constructed by using $u_\bullet$ and $f_\bullet$:
    $$
       \setlength\mathsurround{0pt}
    \begin{tikzcd}
    X_{s+r+r'} \arrow[r,"u_{s+r+r'}"] & Y_{s+r+r'} \\
    S^0 \arrow[u,"f_{s+r+r'}"] \arrow[r,"id"] & S^0 \arrow[u,dashrightarrow, "g_1"]
    \end{tikzcd}.
    $$
    The condition $i_*(x')=y'_2$ guarantees that $g_*(b_{Y,2})=y'_2$
    
To construct $g_2$, notice that in this part, $D_Y$ agrees with $D(s+r',r)$.  By the universal property of $D(s+r',r)$, there exists $g_2$ such that $g_*(a_{Y,2})=y_2$. This map is determined up to a choice of the lift of $y_2$.
    
The map $g_3$ is determined by the following diagram:
$$
   \setlength\mathsurround{0pt}
\begin{tikzcd}
&&&Y_{s+r'}\arrow{dd}{\alpha^Y_{s+r}}\\
S^0 \arrow[hookrightarrow]{rr} \arrow[hook]{dd}& & D^{1} \arrow{ur}{g_2} & \\
& X_{s+r} \arrow{rr}{u_{s+r}\hspace{0.5in}} & & Y_{s+r} \\
D^{1} \arrow[hookrightarrow,"i^-"]{rr} \arrow{ur}{f_{s+r}} && S^{1} \arrow[dashrightarrow,"g_3"]{ru} \arrow[hook,from=uu,"i^+",crossing over]&
\end{tikzcd}$$
    
Since $S^{1}=D^{1} \cup D^{1}$, the map $g_3 \colon S^{1} \rightarrow Y_{s+r}$ is determined by the maps
$$D^{1} \xrightarrow{i^+} Y_{s+r'} \xrightarrow{\alpha_{s+r}^Y} Y_{s+r},$$
$$D^{1} \xrightarrow{i^-} X_{s+r} \xrightarrow{\alpha_{s+r}^Y} Y_{s+r}.$$
The condition $i_*(x)=y'_1$ guarantees that $g_*(b_{Y,1})=y'_1$.
    
    Finally, a choice of the lift of $y_1$ would produce $g_4$ by the universal property of $D(s,r)$.  This completes the proof of the theorem.   

\end{proof} 
\bibliography{pic}\bibliographystyle{alpha}
\end{document}